\documentclass[11pt,a4paper]{article}

\usepackage{a4wide}

\usepackage{graphicx}
\usepackage{subcaption}
\usepackage{amssymb,amsmath,amsthm,mathrsfs,xfrac}
\usepackage{enumerate,color}
\usepackage{hyperref}
\usepackage{array}
\usepackage{amsfonts}
\usepackage{psfrag}
\usepackage{pgf,tikz}
\usetikzlibrary{arrows}
\usepackage{mathtools}
\newtheorem{theorem}{Theorem}

\newtheorem{definition}{Definition}
\newtheorem{remark}{Remark}

\newtheorem{lemma}{Lemma}
\newtheorem{proposition}{Proposition}

\numberwithin{equation}{section}
\allowdisplaybreaks[4]

\renewenvironment{proof}{\smallskip\noindent\emph{\textbf{Proof.}}%
  \hspace{1pt}}{\hspace{-5pt}{\nobreak\quad\nobreak\hfill\nobreak%
    $\square$\vspace{2pt}\par}\smallskip\goodbreak}

\newenvironment{proofof}[1]{\smallskip\noindent{\textbf{Proof~of~#1.}}%
  \hspace{1pt}}{\hspace{-5pt}{\nobreak\quad\nobreak\hfill\nobreak%
    $\square$\vspace{2pt}\par}\smallskip\goodbreak}

\newcommand{\C}[1]{\mathbf{C^{#1}}}
\newcommand{\Cc}[1]{\mathbf{C}_c^{#1}}

\renewcommand{\L}[1]{\mathbf{L^#1}}
\newcommand{\Lloc}[1]{\mathbf{L^{#1}_{loc}}}
\newcommand{\W}[2]{{\mathbf{W}^{#1,#2}}}

\newcommand{\modulo}[1]{{\left|#1\right|}}
\newcommand{\norma}[1]{{\left\|#1\right\|}}
\newcommand{\caratt}[1]{{\chi_{\strut#1}}}
\newcommand{\R}{\mathbb{R}}
\newcommand{\reali}{\mathbb{R}}
\newcommand{\N}{\mathbb{N}}

\newcommand{\Z}{\mathbb{Z}}
\newcommand{\interi}{\mathbb{Z}}
\renewcommand{\epsilon}{\varepsilon}
\renewcommand{\phi}{\varphi}
\renewcommand{\theta}{\vartheta}

\newcommand{\tv}{\mathinner{\rm TV}}

\renewcommand{\d}[1]{\mathinner{\mathrm{d}{#1}}}

\newcommand{\del}{\partial}

\newcommand{\be}{\begin{equation}}
\newcommand{\ee}{\end{equation}}

\definecolor{ffqqqq}{rgb}{1.,0.,0.}
\definecolor{uuuuuu}{rgb}{0.26666666666666666,0.26666666666666666,0.26666666666666666}

\DeclareMathOperator{\sgn}{sgn}

\makeatletter
\let\@fnsymbol\@arabic
\makeatother

\graphicspath{{./figure/}}

\title{Stability estimates for non-local scalar conservation laws}

\author{
\textsc{Felisia Angela Chiarello\footnotemark[1]}
\and
\textsc{Paola Goatin\footnotemark[1]}
\and
\textsc{Elena Rossi\footnotemark[1]}
}

\date{ }

\begin{document}
\maketitle

\footnotetext[1]{Inria Sophia Antipolis - M\'editerran\'ee,
  Universit\'e C\^ote d'Azur, Inria, CNRS, LJAD, 2004 route des
  Lucioles - BP 93, 06902 Sophia Antipolis Cedex, France. E-mail:
  \texttt{\{felisia.chiarello, paola.goatin, elena.rossi\}@inria.fr}}

\begin{abstract}
  \noindent

  We prove the stability of entropy weak solutions of a class of
  scalar conservation laws with non-local flux arising in traffic
  modelling. We obtain an estimate of the dependence of the solution
  with respect to the kernel function, the speed and the initial
  datum. Stability is obtained from the entropy condition through
  doubling of variable technique. We finally provide some numerical
  simulations illustrating the dependencies above for some cost
  functionals derived from traffic flow applications.

  \medskip

  \noindent\textit{2010~Mathematics Subject Classification: 35L65, 35L60, 35L04, 90B20 }

  \medskip

  \noindent\textbf{Key words:} Scalar conservation laws; Non-local
  flux; Stability.

\end{abstract}




\section{Introduction}
\label{sec:Intro}

Conservation laws with non-local flux have drawn growing attention in
the recent years.  Indeed, beside the intrinsic mathematical interest
for their properties, they turned out to be suitable for modelling
several phenomena arising in natural or engineering sciences: flux
functions depending on space-integrals of the unknown appear for
example in models for granular flows \cite{AmadoriShen2012},
sedimentation \cite{Betancourt2011}, supply chains \cite{Keimer2014},
conveyor belts \cite{Gottlich2014}, weakly coupled oscillators
\cite{AHP2016}, structured populations
dynamics~\cite{Perthame_book2007} and traffic
flows~\cite{BlandinGoatin2016, ColomboGaravelloMercier2012,
  SopasakisKatsoulakis2006}.

For this type of equations, general existence and uniqueness results
have been established in~\cite{AmorimColomboTexeira, ChiarelloGoatin} for specific
classes of scalar equations in one space-dimension, and
in~\cite{ACG2015} for multi-dimensional systems of equations coupled
through the non-local term. In particular, existence is usually proved
by providing suitable compactness estimates on a sequence of
approximate solutions constructed by finite volume schemes, while
$\L1$-stability on initial data is obtained from Kru{\v{z}}kov-type
entropy conditions through the doubling of variable
technique~\cite{Kruzkov}. A different approach based on fixed-point
techniques has been recently proposed in~\cite{KeimerPflug2017} to
prove existence and uniqueness of solutions to scalar balance laws in
one space dimension, whose velocity term depends on the weighted
integral of the density over an area in space.

In this paper, we focus on a specific class of scalar equations, in
which the integral dependence of the flux function is expressed though
a convolution product.  We consider the following Cauchy problem
\begin{equation}\label{eq:1}
  \begin{cases}
    \partial_t \rho + \partial_x \left( f (t,x,\rho) V (t,x) \right)=0
    & t>0, \, x \in \R,
    \\
    \rho(0,x)=\rho_o(x), & x \in \R,
  \end{cases}
\end{equation}
where $V(t,x)=v\left((\rho(t) * w) (x) \right)$, and $w$ is a smooth
mollifier:
\begin{equation*}
  (\rho(t) * w) (x) =\int_{\R} \rho(t,y) \, w(x-y) \d{y}.
\end{equation*}
Here and below, we set $\rho(t):=\rho(t,\cdot)$ the function
$x\mapsto\rho(t,x)$.

Existence and uniqueness of solutions to~\eqref{eq:1} follows
from~\cite{AmorimColomboTexeira}, as well as some \emph{a priori}
estimates, namely $\L1$, $\L\infty$ and total variation estimates, see
Section~\ref{sec:MR} below.

Motivated by the study of control and optimisation problems, we are
interested in studying the dependence of solutions to~\eqref{eq:1} on
the convolution kernel $w$ and on the velocity function $v$. Estimates
of the dependence of solutions to a general balance laws on the flux
function can be found
in~\cite{ColomboMercierRosini,MercierV2}. However, as precised also
below (see Remark~\ref{rem:Magali}), those estimates turn out
to be implicit when applied to the setting of
  problem~\eqref{eq:1}.

Carefully applying the Kru{\v{z}}kov's doubling of variables
techniques, on the lines of~\cite{Betancourt2011, KarlsenRisebro}, we
derive the $\L1$-Lipschitz continuous dependence of solutions to
\eqref{eq:1} on the initial datum, the kernel (see
Theorem~\ref{thm:stabw}) and the velocity (see
Theorem~\ref{thm:stabv}).  These results are collected in
Section~\ref{sec:MR}, while the technical proofs are deferred to
Section~\ref{sec:proofs}.  Finally, in Section~\ref{sec:numint} we
show some numerical simulation illustrating the behaviour of the
solutions of a non-local traffic flow model, when the size and the
position of the kernel support or the velocity function vary.  In
particular, we analyse the impact on two cost functionals, measuring
traffic congestion.

\section{Main Results}
\label{sec:MR}

The study of problem~\eqref{eq:1} is carried out in the same setting
of~\cite{AmorimColomboTexeira}, with slightly strengthened
conditions. We recall here briefly the assumptions on the flux
function $f$, on $v$ and on $w$:
\begin{align}
  \label{eq:hpf}
  f \in \C2 ( \R \times \R \times \R; \R^+)
  \quad
  \mbox{and}
  &
    \quad
    \left\{
    \begin{array}{r@{\,}c@{\,}l}
      \displaystyle \sup_{t,x,\rho} \modulo{\partial_\rho f(t,x,\rho)}
      & <
      & +\infty
      \\
      \displaystyle  \sup_{t,x}\modulo{\partial_x f(t,x,\rho)}
      &<
      &C \modulo{\rho}
      \\
      \displaystyle \sup_{t,x}\modulo{\partial^2_{xx} f(t,x,\rho)}
      &<
      &C  \modulo{\rho}
      \\
      \forall t, \, x \quad f (t,x,0)
      & =
      & 0
    \end{array}
        \right.
  \\
  \label{eq:hpvw}
  v \in (\C2 \cap \W2\infty)(\R;\R)
  \quad
  \mbox{and}
  &
    \quad
    w\in (\C2 \cap \W11 \cap \W2\infty)(\R;\R).
\end{align}

We recall also the definition of solution to problem~\eqref{eq:1},
see~\cite[Definition~2.1]{AmorimColomboTexeira}.
\begin{definition}
  \label{def:sol}
  Let $T > 0$. Fix $\rho_o \in \L\infty (\R;\R)$. A \emph{weak entropy
    solution} to~\eqref{eq:1} on $[0,T]$ is a bounded measurable Kru\v
  zkov solution $\rho \in \C0 ([0,T]; \Lloc1 (\R;\R))$ to
  \begin{displaymath}
    \left\{
      \begin{array}{l}
        \partial_t \rho +
        \partial_x \left( f (t,x,\rho)  \, V (t,x) \right) = 0
        \\
        \rho (0,x) = \rho_o (x)
      \end{array}
    \right.
    \quad
    \mbox{where}
    \quad
    V (t,x) = v ((\rho (t)*w) (x)).
  \end{displaymath}
\end{definition}

The results in~\cite{AmorimColomboTexeira} ensure the existence and
uniqueness of solution to~\eqref{eq:1} and provides the following
\emph{a priori} estimates on the solution.
\begin{lemma}[{\cite[Lemma~2.2]{AmorimColomboTexeira}}]
  \label{lem:pos}
  Let conditions~\eqref{eq:hpf}-\eqref{eq:hpvw} hold. If
  $\rho_o (x) \geq 0$ for all $x \in \R$, then the solution
  to~\eqref{eq:1} is such that $\rho(t,x)\geq 0$ for all
  $(t,x)\in \R^+\times \R.$
\end{lemma}

\begin{lemma}[{\cite[Lemma~2.4]{AmorimColomboTexeira}}]
  \label{lem:L1}
  Let conditions~\eqref{eq:hpf}-\eqref{eq:hpvw} hold. If
  $\rho_o (x) \geq 0$ for all $x \in \R$, then the solution
  to~\eqref{eq:1} satisfies, for all $t \in \R^+$,
  \begin{displaymath}
    \norma{\rho(t)}_{\L1(\R;\R)}\leq \norma{\rho_o}_{\L1(\R;\R)}.
  \end{displaymath}
\end{lemma}

\begin{lemma}[{\cite[Lemma~2.5]{AmorimColomboTexeira}}]
  \label{lem:Linf}
  Let conditions~\eqref{eq:hpf}-\eqref{eq:hpvw} hold. If
  $\rho_o (x) \geq 0$ for all $x \in \R$, then the solution
  to~\eqref{eq:1} satisfies, for all $t \in \R^+$,
  \begin{equation}
    \label{eq:Linf}
    \norma{\rho(t)}_{\L\infty(\R;\R)}\leq
    \norma{\rho_o}_{\L\infty(\R;\R)}\, e^{\mathcal{L}t},
  \end{equation}
  where
  $\mathcal{L}=C\,\norma{v}_{\L\infty(\R;\R)} + \norma{\partial_\rho
    f}_{\L\infty([0,t] \times \R
    \times\R;\R)}\,\norma{v'}_{\L\infty(\R;\R)}\,
  \norma{\rho_o}_{\L1(\R;\R)}\,\norma{w'}_{\L\infty(\R;\R)}.$
\end{lemma}

\begin{proposition}[{\cite[Proposition~2.6]{AmorimColomboTexeira}}]
  \label{prop:tv}
  Let conditions~\eqref{eq:hpf}-\eqref{eq:hpvw} hold. If
  $\rho_o (x) \geq 0$ for all $x \in \R$, then the solution
  to~\eqref{eq:1} satisfies the following total variation estimate:
  for all $t \in \R^+$
  \begin{equation}
    \label{eq:tv}
    \tv(\rho(t))\leq\left(\mathcal K_2 \,t+\tv(\rho_o)\right) e^{\mathcal K_1\,t},
  \end{equation}
  where
  \begin{align}
    \nonumber
    \mathcal K_1= \
    &  \norma{\partial^2_{\rho x} f }_{\L\infty(\Sigma_t;\R)} \, \norma{v}_{\L\infty(\R;\R)},
    \\
    \label{eq:2}
    \mathcal K_2= \
    & \left[\frac{3}{2} C + \left(
      \norma{\partial_\rho f }_{\L\infty(\Sigma_t;\R)}+C \right) \norma{w'}_{\W1\infty(\R;\R)}\norma{\rho_o}_{\L1(\R;\R)}
      \right.
    \\
    \nonumber
    &
      \left. + \frac{1}{2}\left(
      C+\norma{\partial_\rho f}_{\L\infty(\Sigma_t;\R)}
      \left( 2 + \norma{\rho_o}_{\L1(\R;\R)} \norma{w'}_{\L\infty(\R;\R)}\right)
      \right)
      \norma{w'}_{\W1\infty(\R;\R)}
      \right]
    \\
    \nonumber
    & \times
      \norma{v}_{\W2\infty (\R;\R)} \norma{\rho_o}_{\L1 (\R;\R)},
  \end{align}
  with $\Sigma_t = [0,t] \times \R \times [0,M_t]$ and
  $M_t = \norma{\rho_o}_{\L\infty(\R;\R)}\, e^{\mathcal{L}t}$, as
  in~\eqref{eq:Linf}.
\end{proposition}



\begin{remark} {\rm { The regularity assumptions required
      in~\cite{AmorimColomboTexeira} for the functions $v$ and $w$,
      see~\cite[Formula~(2.2)]{AmorimColomboTexeira}, are actually
      less restrictive than~\eqref{eq:hpvw}. Indeed, to guarantee the
      existence of solutions and to obtain the \emph{a priori}
      estimates above, it is sufficient that
  \begin{displaymath}
    v \in (\C2 \cap \W1\infty) (\reali;\reali)
    \qquad \mbox{and} \qquad
    w \in (\C2 \cap \W2\infty) (\reali;\reali).
  \end{displaymath}}}
\end{remark}

\smallskip Aim of this paper is to study the stability of solutions
to~\eqref{eq:1} with respect to both the kernel $w$ and the velocity
function $v$. The following Theorem states the $\L1$--Lipschitz
continuous dependence of solutions to~\eqref{eq:1} on both the initial
datum and the kernel function.
\begin{theorem}
  \label{thm:stabw}
  Let $T > 0$. Fix $f$ and $v$ satisfying~\eqref{eq:hpf}
  and~\eqref{eq:hpvw} respectively.  Fix
  $\rho_o, \tilde \rho_o \in \L\infty (\R;\R)$. Let
  $w, \tilde w \in (\C2 \cap \W11 \cap \W2\infty) (\R; \R)$. Call
  $\rho$ and $\tilde \rho$ the solutions, in the sense of
  Definition~\ref{def:sol}, to the following problems respectively
  \begin{align}
    \label{eq:rho}
    \left\{
    \begin{array}{l}
      \partial_t \rho +
      \partial_x ( f (t,x,\rho)  \, V (t,x) ) = 0
      \\
      \rho (0,x) = \rho_o (x)
    \end{array}
    \right.
    \quad
    \mbox{where}
 &
   \quad
   V (t,x) = v ((\rho (t)*w) (x)),
    \\
    \label{eq:rhot}
    \left\{
    \begin{array}{l}
      \partial_t \tilde \rho +
      \partial_x ( f (t,x,\tilde \rho)  \, \tilde V (t,x) ) = 0
      \\
      \tilde \rho (0,x) = \tilde \rho_o (x)
    \end{array}
    \right.
    \quad
    \mbox{where}
 &
   \quad
   \tilde  V (t,x) = v ((\tilde \rho (t)*\tilde w) (x)).
  \end{align}
  Then, for any $t \in [0,T]$, the following estimate holds
  \begin{equation}
    \label{eq:3}
    \norma{\rho(t)-\tilde \rho(t)}_{\L1(\R; \reali)}
    \leq
    \left( \norma{\rho_o-\tilde{\rho}_o}_{\L1(\R;\reali)} + a(t) \, \norma{w-\tilde w}_{\W11(\R;\R)}  \right) \exp\left(\int_0^t b (r) \d{r} \right),
  \end{equation}
  where $a(t)$ and $b (t)$ depend on various norms of the initial data
  and of the functions $f$, $v$, $w$ and $\tilde w$, see~\eqref{eq:a}
  and~\eqref{eq:b}.
\end{theorem}

The $\L1$--Lipschitz continuous dependence of solutions to~\eqref{eq:1}
on the velocity function $v$ is ensured by the following Theorem.
\begin{theorem}
  \label{thm:stabv}
  Let $T > 0$. Fix $f$ and $w$ satisfying~\eqref{eq:hpf}
  and~\eqref{eq:hpvw} respectively. Fix $\rho_o \in \L\infty
  (\R;\R)$. Let $v, \tilde v \in (\C2 \cap \W2\infty) (\R; \R)$. Call
  $\rho$ and $\tilde \rho$ the solutions, in the sense of
  Definition~\ref{def:sol}, to the following problems respectively
  \begin{align}
    \label{eq:rho1}
    \left\{
    \begin{array}{l}
      \partial_t \rho +
      \partial_x ( f (t,x,\rho)  \, V (t,x) ) = 0
      \\
      \rho (0,x) = \rho_o (x)
    \end{array}
    \right.
    \quad
    \mbox{where}
 &
   \quad
   V (t,x) = v ((\rho (t)*w) (x)),
    \\
    \label{eq:rhot1}
    \left\{
    \begin{array}{l}
      \partial_t \tilde \rho +
      \partial_x ( f (t,x,\tilde \rho)  \, \tilde V (t,x) ) = 0
      \\
      \tilde \rho (0,x) = \rho_o (x)
    \end{array}
    \right.
    \quad
    \mbox{where}
 &
   \quad
   \tilde  V (t,x) = \tilde v ((\tilde \rho (t)*w) (x)).
  \end{align}
  Then, for any $t \in [0,T]$, the following estimate holds
  \begin{equation}
    \label{eq:16}
    \norma{\rho(t)-\tilde \rho(t)}_{\L1(\R; \reali)}
    \leq  \left(
      c_1 (t) \, \norma{v - \tilde v}_{\L\infty (\reali;\reali)}
      +
      c_2 (t) \, \norma{v' - \tilde v'}_{\L\infty (\reali;\reali)}
    \right) \, \exp\left(\int_0^t c_3 (s) \d{s} \right),
  \end{equation}
  where the $c_i(t)$, $i = 1, 2, 3$, depend on various norms of the
  initial data and of the functions $f$, $v$, $\tilde v$ and $w$,
  see~\eqref{eq:c1}, \eqref{eq:c2} and~\eqref{eq:c3}.
\end{theorem}

\section{Proofs}
\label{sec:proofs}

The Lemma below is the building block of both Theorem~\ref{thm:stabw}
and Theorem~\ref{thm:stabv}.
\begin{lemma}
  \label{lem:dv}
  Let $T > 0$. Fix $f$ satisfying~\eqref{eq:hpf} and
  $V, \, \tilde V \in (\C2 \cap \W2\infty) (\reali \times \reali;
  \reali)$. Fix $\rho_o, \tilde \rho_o \in \L\infty (\R;\R)$. Call
  $\rho$ and $\tilde \rho$ the solutions to the following problems
  \begin{align}
    \label{eq:17}
    \left\{
    \begin{array}{l}
      \partial_t \rho +
      \partial_x ( f (t,x,\rho)  \, V (t,x) ) = 0
      \\
      \rho (0,x) = \rho_o (x)
    \end{array}
    \right.
    \quad
    \mbox{and}
    \quad
    \left\{
    \begin{array}{l}
      \partial_t \tilde \rho +
      \partial_x ( f (t,x,\tilde \rho)  \, \tilde V (t,x) ) = 0
      \\
      \tilde \rho (0,x) = \tilde \rho_o (x).
    \end{array}
    \right.
  \end{align}
  Then, for any $\tau, t \in \, ]0,T[$, with $\tau < t$, the following
  estimate holds
  \begin{align}
    \label{eq:19}
    \int_{\R} \modulo{\rho (\tau,x) - \tilde \rho (\tau,x)} \d{x}
    -\int_{\R} \modulo{\rho (t,x) - \tilde \rho (t,x)} \d{x}
    &
    \\ \nonumber
    + \int_{\tau}^{t} \! \int_{\R}
    \big\{\modulo{\partial_x \tilde V (s,x)  - \partial_x V (s,x)} \, \modulo{f \left(s,x,\rho (s,x) \right)}
    &
    \\ \nonumber
    + \modulo{\tilde V (s,x) -  V (s,x)} \,\modulo{\partial_x f \left(s,x, \rho (s,x) \right)}
    &
    \\[2pt] \label{eq:20}
    + \modulo{\tilde V (s,x) - V (s,x)} \,
    \modulo{\partial_\rho f  \left(s,x,\rho (s,x) \right) } \,
    \modulo{\partial_x \rho (s,x) }\big\}
    & \d{x}\d{s} \geq \   0.
  \end{align}

\end{lemma}

\begin{proof}
  The proof is based on the \emph{doubling of variables method}
  introduced by Kru\v zkov in~\cite{Kruzkov}. In particular, we follow
  the lines of~\cite[Theorem~1.3]{KarlsenRisebro}, although there the
  flux function has the form $l (x) \, g (\rho)$, while here it is of
  the form $f (t,x,\rho) \, V (t,x)$. The dependence on time does not
  add any difficulties in the proof, while the dependence of $f$ on
  the space variable $x$ produces additional terms.

  Let $\phi \in \Cc\infty (]0;T[ \times \reali; \reali^+)$ be a test
  function as in the definition of solution by Kru\v zkov. Let
  $Y \in \Cc\infty (\reali;\reali^+)$ be such that
  \begin{align*}
    Y(z) = \
    & Y(-z),
    &
      Y(z) = \
    &0 \mbox{ for }  \modulo{z}\geq 1,
    &
      \int_{\R} Y(z) \d{z} = \
    &1,
  \end{align*}
  and define $Y_h =\frac1h Y\! \left(\frac{z}{h}\right)$. Obviously
  $Y_h \in \Cc\infty (\reali; \reali^+)$, $Y_h (-z)=Y_h (z)$,
  $Y_h (z) = 0$ for $\modulo{z}\geq h$, $\int_\reali Y_h (z)\d{z}=1$
  and $Y_h \to \delta_0$ as $h \to 0$, where $\delta_0$ is the Dirac
  delta in $0$. Define, for $h>0$,
  \begin{equation}
    \label{eq:psih}
    \psi_h (t,x,s,y) =
    \phi \left(\frac{t+s}{2},\frac{x+y}{2}\right) \, Y_h (t-s) \, Y_h (x-y)
    =
    \phi \left(\cdots\right) \, Y_h (t-s) \, Y_h (x-y).
  \end{equation}
  Introduce the space $\Pi_T=\, ]0,T[\times\reali$. We derive the
  following entropy inequalities for the solutions $\rho=\rho(t,x)$
  and $\tilde{\rho}=\tilde{\rho}(s,y)$
  to~\eqref{eq:17}:
  \begin{align*}
    \iiiint\limits_{\Pi_T\times\Pi_T}
    \left\{\modulo{\rho-\tilde{\rho}} \, \partial_t\psi_h (t,x,s,y)
    + \sgn(\rho-\tilde{\rho}) V(t,x) \left( f(t,x,\rho)-f(t,x,\tilde \rho)\right)\, \partial_x\psi_h (t,x,s,y)\right.
    &
    \\
    \left. + \sgn(\rho-\tilde{\rho}) \, \partial_x \left[ f(t,x,\tilde{\rho}) \,V(t,x)\right] \psi_h  (t,x,s,y)
    \right\}\d{x}\d{t}\d{y}\d{s}
    & \geq 0
  \end{align*}
  and
  \begin{align*}
    \iiiint\limits_{\Pi_T\times\Pi_T}
    \left\{\modulo{\tilde{\rho} - \rho} \, \partial_t\psi_h (t,x,s,y)
    + \sgn(\tilde{\rho} - \rho) \tilde V(s,y) \left( f(s,y,\tilde\rho)-f(s,y, \rho)\right)\, \partial_y\psi_h (t,x,s,y)\right.
    &
    \\
    \left. + \sgn(\tilde\rho-\rho) \, \partial_y \left[ f(s,y,\rho) \, \tilde V(s,y)\right] \psi_h  (t,x,s,y)
    \right\}\d{x}\d{t}\d{y}\d{s}
    & \geq 0.
  \end{align*}
  Summing the two inequalities above and rearranging the terms
  therein, relying on the explicit form of the function
  $\psi_h$~\eqref{eq:psih}, we obtain
  \begin{align}
    \label{eq:5}
    \iiiint\limits_{\Pi_T \times\Pi_T}
    \biggl\{\modulo{\rho (t,x) -\tilde \rho (s,y)} \,
    \partial_t\phi  \left(\cdots\right) \, Y_h (t-s) \, Y_h (x-y)
    &
    \\
    \label{eq:I1}
    + \sgn(\rho -\tilde \rho) \left(
    V(t,x) f(t,x,\rho) - \tilde V(s,y) f(s,y,\tilde\rho)\right)
    \partial_x \phi  \left(\cdots\right) \, Y_h (t-s) \, Y_h (x-y)
    &
    \\
    \label{eq:I21a}
    + \sgn(\rho -\tilde \rho) \left(
    \tilde V(s,y) f(s,y,\tilde\rho) - V(t,x) f(t,x,\tilde\rho)\right)
    \partial_x \psi_h (t,x,s,y)
    &
    \\
    \label{eq:I21b}
    + \sgn(\rho -\tilde \rho) \left(
    \tilde V(s,y) f(s,y,\rho) - V(t,x) f(t,x,\rho)\right)
    \partial_y \psi_h (t,x,s,y)
    &
    \\
    \label{eq:I22}
    \left. + \sgn(\rho-\tilde \rho) \left[
    \partial_y \left( \tilde V(s,y)\, f(s,y,\rho) \right)
    -
    \partial_x \bigl( V(t,x) \, f(t,x,\tilde \rho)  \bigr)
    \right]
    \psi_h  (t,x,s,y)
    \right\}
    &
    \\
    \nonumber
    \d{x}\d{t}\d{y}\d{s}
    & \geq 0.
  \end{align}
  Consider~\eqref{eq:I21a} and~\eqref{eq:I21b}: explicit the function
  $\psi_h$ to obtain
  \begin{align}
    \nonumber
    & [\eqref{eq:I21a}] + [\eqref{eq:I21b}]
    \\
    \label{eq:I211}
    = \
    & \frac{ \sgn(\rho-\tilde \rho)}2
      \tilde V(s,y) \left(f(s,y,\tilde \rho) + f (s,y,\rho)\right)
      \partial_x \phi (\cdots) \, Y_h (t-s) \, Y_h (x-y)
    \\
    \label{eq:I211bis}
    & - \frac{ \sgn(\rho-\tilde \rho)}2
      V(t,x) \left(f(t,x,\tilde \rho) + f (t,x,\rho)\right)
      \partial_x \phi (\cdots) \, Y_h (t-s) \,Y_h (x-y)
    \\
    \label{eq:I212}
    & - \sgn(\rho-\tilde \rho)
      \tilde V (s,y)  \left(f(s,y,\tilde \rho) - f (s,y,\rho)\right)
      \phi (\cdots) \, Y_h (t-s) \, Y_h' (x-y)
    \\
    \label{eq:I212bis}
    & + \sgn(\rho-\tilde \rho)
      V (t,x)  \left(f(t,x, \rho) - f (t,x,\tilde \rho)\right)
      \phi (\cdots) \, Y_h (t-s) \, Y_h' (x-y).
  \end{align}
  In~\eqref{eq:I22} compute
  \begin{equation}
    \label{eq:9}
    \begin{aligned}
      [\eqref{eq:I22}] = \ & \sgn (\rho - \tilde \rho) \left[
        \partial_y \tilde V (s,y) \, f (s,y,\rho) + \tilde V (s,y)
        \, \partial_y f (s,y,\rho) \right.
      \\
      &\qquad -
      \partial_x V (t,x) \, f (t,x,\tilde \rho) - V (t,x)
      \, \partial_x f (t,x,\tilde\rho) \bigr] \psi_h (t,x,s,y).
    \end{aligned}
  \end{equation}
  Introduce the following notation
  \begin{equation}
    \label{eq:F}
    F \left(t,x,\rho (t,x), \tilde \rho(s,y)\right) \! =\!
    \sgn \left(\rho(t,x) - \tilde \rho (s,y)\right) \!\left(
      f \left(t,x,\rho(t,x)\right) - f \left(t,x,\tilde\rho(s,y)\right)\!
    \right),
  \end{equation}
  so that~\eqref{eq:I212} -- \eqref{eq:I212bis} now reads
  \begin{align}
    \nonumber
    & \iiiint\limits_{\Pi_T\times\Pi_T} [\eqref{eq:I212}] +  [\eqref{eq:I212bis}] \d{x}\d{t}\d{y}\d{s}
    \\
    \nonumber
    = \
    & \iiiint\limits_{\Pi_T\times\Pi_T}
      \left( V (t,x) F (t,x,\rho,\tilde \rho)
      -  \tilde V(s,y) F (s,y,\rho,\tilde \rho)\right)
      \phi \!\left(\cdots\right)  Y_h (t-s) \, Y_h' (x-y)
      \d{x}\d{t}\d{y}\d{s}
    \\
    \label{eq:q1}
    = \
    &
      - \!\!\iiiint\limits_{\Pi_T\times\Pi_T}\!
      \left( V (t,x) \frac{\d{}}{\d{x}} F (t,x,\rho,\tilde \rho)
      -  \tilde V(s,y) \frac{\d{}}{\d{x}} F (s,y,\rho,\tilde \rho)\right)
      \psi_h(t,x,s,y)
      \d{x}\d{t}\d{y}\d{s}
    \\
    \label{eq:q2}
    & - \!\!\iiiint\limits_{\Pi_T\times\Pi_T}\!
      \partial_x V (t,x) \, F (t,x,\rho,\tilde \rho) \,
      \psi_h (t,x,s,y)
      \d{x}\d{t}\d{y}\d{s}
    \\
    \label{eq:J}
    & - \!\!\iiiint\limits_{\Pi_T\times\Pi_T}  \frac12
      \left( V (t,x) F (t,x,\rho,\tilde \rho)
      -  \tilde V(s,y) F (s,y,\rho,\tilde \rho)\right)
      \partial_x \phi \!\left(\cdots\right)
    \\
    \nonumber
    & \qquad \qquad \qquad\qquad\qquad\qquad\qquad \times Y_h (t-s) \, Y_h (x-y)
      \d{x}\d{t}\d{y}\d{s},
  \end{align}
  where we also integrate by parts. Combine the integrand
  of~\eqref{eq:q2} together with~\eqref{eq:9} to get
  \begin{align}
    \nonumber
    & - \partial_x V (t,x) \, F (t,x,\rho,\tilde \rho) \, \psi_h (t,x,s,y) + [\eqref{eq:9}]
    \\
    \label{eq:nonloso}
    = \
    & \sgn (\rho - \tilde \rho)
      \left(\partial_y \tilde V (s,y) \, f (s,y,\rho) -\partial_x V (t,x)\, f (t,x,\rho) \right) \, \psi_h (t,x,s,y)
    \\\label{eq:nonloso2}
    &
      + \sgn (\rho - \tilde \rho) \left( \tilde V (s,y) \, \partial_y f (s,y,\rho) - V (t,x) \partial_x f (t,x,\tilde\rho) \right)  \, \psi_h (t,x,s,y).
  \end{align}
  Observe that the following equality holds
  \begin{align}
    \nonumber
    & \iiiint\limits_{\Pi_T\times\Pi_T}
      [\eqref{eq:I1}]+[\eqref{eq:I211}]+[\eqref{eq:I211bis}]+[\eqref{eq:J}]
      \d{x}\d{t}\d{y}\d{s}
    \\
    \label{eq:4}
    = \
    &
      \iiiint\limits_{\Pi_T\times\Pi_T} \!\sgn (\rho -\tilde \rho) \,
      \tilde V (s,y)\! \left(f (s,y,\rho) - f (s,y,\tilde\rho)\right)
      \partial_x \phi (\cdots) \, Y_h (t-s) \, Y_h (x-y)
      \d{x}\d{t}\d{y}\d{s}.
  \end{align}
  We are therefore left with
  \begin{equation}
    \label{eq:partenza}
    \iiiint\limits_{\Pi_T\times\Pi_T}\!  [\eqref{eq:5}] \!+\!
    [\eqref{eq:q1}] \!+\! [\eqref{eq:nonloso}] \!+\! [\eqref{eq:nonloso2}]
    \!+\! [\eqref{eq:4}]
    \d{x}\d{t}\d{y}\d{s}
    \geq 0.
  \end{equation}
  Let now $h$ go to $0$. The terms in~\eqref{eq:5} and~\eqref{eq:4}
  can be treated exactly as in~\cite{Kruzkov}, leading to
  \begin{align}
    \nonumber
    & \lim_{h \to 0+}\iiiint\limits_{\Pi_T \times \Pi_T}
      \left\{[\eqref{eq:5}] \!+\! [\eqref{eq:4}]
      \right\}\d{x}\d{t}\d{y}\d{s}
    \\ \label{eq:h1}
    \!\!= \
    & \iint\limits_{\Pi_T} \bigl\{
      \modulo{\rho (t,x) - \tilde \rho (t,x)} \partial_t \phi (t,x)\\ \label{eq:h2}
    & + \sgn \left(\rho (t,x) - \tilde\rho (t,x)\right) \tilde V (t,x)
      \left(f \left(t,x,\rho (t,x)\right) - f\left(t,x,\tilde\rho (t,x) \right)\right) \partial_x\phi (t,x)\bigr\}
      \d{x}\d{t}.
  \end{align}
  Regarding~\eqref{eq:nonloso},  we simplify the notation by introducing
  the map
  \begin{displaymath}
    \Upsilon (t,x,s,y) \! = \!
    \sgn\left(\rho (t,x) \! - \!\tilde \rho (s,y)\right) \!\!
    \left(\partial_y \tilde V (s,y) f \left(s,y, \rho \right) - \partial_x V (t,x) f \left(t,x,\rho\right)\right)
    \phi\! \left(\frac{t+s}2, \frac{x+y}2\right)\!\!,
  \end{displaymath}
  so that
  \begin{align}
    \nonumber
    & [\eqref{eq:nonloso}]
    \\ \nonumber
    = \
    &\Upsilon (t,x,s,y) \, Y_h (t-s) \, Y_h(x-y)
    \\ \nonumber
    =\
    & \Upsilon (t,x,t,x) \, Y_h (t-s) \, Y_h (x-y)
      + \left(\Upsilon (t,x,s,y) - \Upsilon (t,x,t,x) \right)\, Y_h (t-s) \, Y_h (x-y)
    \\
    \label{eq:10}
    = \
    &  \sgn\left(\rho (t,x)  -\tilde \rho (t,x) \right)
      \left(\partial_x \tilde V (t,x)  - \partial_x V (t,x) \right) f \left(t,x,\rho \right)
      \phi\! \left(t,x\right)  \, Y_h (t-s) \, Y_h (x-y)
    \\
    \label{eq:10bis}
    &+
      \left(\Upsilon (t,x,s,y) - \Upsilon (t,x,t,x) \right)\, Y_h (t-s) \, Y_h (x-y) .
  \end{align}
  It is immediate to see that
  \begin{align}
    \nonumber
    & \iiiint\limits_{\Pi_T\times\Pi_T} [\eqref{eq:10}] \d{x}\d{t}\d{y}\d{s}
    \\ \label{eq:10ok}
    = \
    & \iint\limits_{\Pi_T} \sgn\left(\rho (t,x)  -\tilde \rho (t,x)\right)
      \left( \partial_x \tilde V (t,x)  - \partial_x V (t,x) \right)  f \left(t,x,\rho (t,x)\right) \,
      \phi\! \left(t,x\right) \d{x}\d{t}.
  \end{align}
  Concerning~\eqref{eq:10bis}, it vanishes as $h$ goes to $0$ when
  integrated over $\Pi_T \times \Pi_T$. Indeed, recall that
  $\modulo{Y_h} \leq \left(Y (0)/h\right) \caratt{[-h,h]}$ and
  apply~\cite[Lemma~6.2]{ColomboRossiBordo}, see
  also~\cite[Lemma~2]{Kruzkov}, with $N=3$, $X= (x,t,x)$, $Y= (x,t,y)$
  and
  \begin{displaymath}
    w (s,Y) = \frac{\left(Y (0)\right)^2}{h^2} \, \Upsilon (t,x,s,y).
  \end{displaymath}
  Focus the attention on~\eqref{eq:q1}. With abuse of notation, since
  the function $F$ is only Lipschitz continuous with respect to
  $\rho$, we write
  \begin{align}
    \nonumber
    &  \frac{\d{}}{\d{x}}F  \left(t,x,\rho (t,x), \tilde \rho (s,y) \right)
    \\
    \nonumber
    = \
    & \partial_x F  \left(t,x,\rho (t,x), \tilde \rho (s,y) \right)
      + \partial_\rho F  \left(t,x,\rho (t,x), \tilde \rho (s,y) \right)   \partial_x \rho (t,x)
    \\
    \label{eq:6}
    = \
    & \sgn\left(\rho (t,x) - \tilde \rho (s,y)\right)
      \left(\partial_x f \left(t,x, \rho (t,x)\right) - \partial_x f \left(t,x, \tilde \rho (s,y)\right)
      \right)
    \\
    \label{eq:6bis}
    & + \partial_\rho F  \left(t,x,\rho (t,x), \tilde \rho (s,y) \right)   \partial_x \rho (t,x)
  \end{align}
  and
  \begin{align}
    \nonumber
    & \frac{\d{}}{\d{x}}  F  \left(s,y,\rho (t,x), \tilde \rho (s,y) \right)
    \\ \nonumber
    =\
    & \partial_\rho F  \left(s,y,\rho (t,x), \tilde \rho (s,y) \right)   \partial_x \rho (t,x)
    \\
    \label{eq:7}
    = \
    & \partial_\rho F  \left(t,x,\rho (t,x), \tilde \rho (t,x) \right)   \partial_x \rho (t,x)
    \\
    \label{eq:7bis}
    & + \left(
      \partial_\rho F  \left(s,y,\rho (t,x), \tilde \rho (s,y) \right)
      - \partial_\rho F  \left(t,x,\rho (t,x), \tilde \rho (t,x) \right)
      \right) \partial_x \rho (t,x).
  \end{align}
  In particular observe that we can combine~\eqref{eq:nonloso2}
  with~\eqref{eq:6} to get
  \begin{align}
    \nonumber
    & [\eqref{eq:nonloso2}] - V (t,x) \sgn\left(\rho (t,x) - \tilde \rho (s,y)\right)\!
      \left(\partial_x f \left(t,x, \rho (t,x)\right)\! - \partial_x f \left(t,x, \tilde \rho (s,y)\right)
      \right) \psi_h (t,x,s,y)
    \\
    \label{eq:11}
    = \
    & \sgn\left(\rho (t,x) - \tilde \rho (s,y)\right) \left(
      \tilde V (s,y) \partial_y f\left(s,y,\rho\right)
      \!-\!
      V (t,x) \partial_x f \left(t,x, \rho\right)
      \right)
      \psi_h (t,x,s,y).
  \end{align}
  An application of~\cite[Lemma~6.2]{ColomboRossiBordo} yields
  \begin{align}
    \nonumber
    &\lim_{h \to 0} \iiiint\limits_{\Pi_T\times \Pi_T} [\eqref{eq:11}] \d{x}\d{t}\d{y}\d{s}
    \\
    \label{eq:11ok}
    = \
    & \iint\limits_{\Pi_T}
      \sgn\left(\rho (t,x) - \tilde \rho (t,x)\right) \left(
      \tilde V (t,x) -  V (t,x) \right) \partial_x f \left(t,x, \rho (t,x)\right)
      \phi (t,x) \d{x}\d{t}.
  \end{align}

  In order to deal with the remaining terms, i.e.~\eqref{eq:6bis},
  \eqref{eq:7} and~\eqref{eq:7bis}, we need to introduce a
  regularisation of the sign function. In particular, for $\alpha > 0$
  set
  \begin{displaymath}
    s_\alpha (u) = \left(\sgn * Y_\alpha\right) (u).
  \end{displaymath}
  Observe that
  $s_\alpha' (u) = \dfrac{2}{\alpha} \, Y\! \left(\dfrac u
    \alpha\right)$.  Recall the definition of the map $F$~\eqref{eq:F}
  and compute
  \begin{align}
    \nonumber
    & \iiiint\limits_{\Pi_T\times \Pi_T} -V(t,x) \times [\eqref{eq:6bis}] \times \psi_h (t,x,s,y) \d{x} \d{t} \d{y}\d{s}
    \\
    \label{eq:nuovo1}
    = \
    & \lim_{\alpha \to 0} \iiiint\limits_{\Pi_T \times \Pi_T}
      \bigl\{ s'_\alpha \left(\rho(t,x)-\tilde \rho(s,y)\right) \left(f \left(t,x,\rho(t,x)\right)-f\left(t,x,\tilde\rho(s,y)\right) \right)
    \\
    \label{eq:nuovo3}
    & \qquad\qquad\quad
      + s_\alpha \left(\rho(t,x)-\tilde \rho(s,y)\right) \partial_\rho f\left(t,x,\rho(t,x)\right)\bigr\}
    \\[2pt]  \label{eq:nuovo}
    & \qquad \qquad
      \times \left( -V(t,x)\right)  \partial_x \rho(t,x) \psi_h (t,x,s,y) \d{x} \d{t} \d{y}\d{s}.
  \end{align}
  By the Dominated Convergence Theorem, as $\alpha$ goes to $0$, we
  get
  \begin{displaymath}
    \iiiint\limits_{\Pi_T\times \Pi_T} [\eqref{eq:nuovo1}] \times [\eqref{eq:nuovo}]\d{x} \d{t} \d{y} \d{s} \to 0.
  \end{displaymath}
  Indeed,
  \begin{align*}
    & \modulo{\frac{2}{\alpha} \, Y\left(\frac{\rho - \tilde \rho}\alpha \right)
      \, \left(f\left(t,x,\rho (t,x)\right) - f \left(t,x,\tilde \rho (s,y)\right)\right)
      V (t,x) \, \partial_x \rho (t,x) \, \psi_h (t,x,s,y)
      }
    \\
    \leq \
    &  \frac{2}{\alpha} \, Y\left(\frac{\rho - \tilde \rho}\alpha \right)
      \int_{\tilde \rho}^\rho \modulo{\partial_\rho f (s,y,r)}\d{r} \,
      \norma{V}_{\L\infty (\Pi_T;\reali)} \,  \modulo{\partial_x \rho (t,x)} \,
      \psi_h (t,x,s,y)
    \\
    \leq \
    & 2 \, \norma{Y}_{\L\infty (\reali;\reali)} \,
      \norma{\partial_\rho f }_{\L\infty (\Pi_T \times \reali;\reali)}
      \norma{V}_{\L\infty (\Pi_T;\reali)} \modulo{\partial_x \rho (t,x)} \,
      \psi_h (t,x,s,y) \quad \in \L1 (\Pi_T \times \Pi_T; \reali).
  \end{align*}
  Therefore we have
  \begin{align}
    \nonumber
    & \iiiint\limits_{\Pi_T\times \Pi_T} -V(t,x) \times [\eqref{eq:6bis}] \times \psi_h (t,x,s,y)  \d{x} \d{t} \d{y} \d{s}
    \\
    \label{eq:13}
    = \
    & \!\! - \!\! \iiiint\limits_{\Pi_T\times \Pi_T} \!\sgn\left(\rho(t,x)-\tilde \rho(s,y)\right) V(t,x) \, \partial_\rho f\left(t,x,\rho(t,x)\right)
      \partial_x \rho(t,x) \, \psi_h (t,x,s,y) \d{x} \d{t} \d{y} \d{s}.
  \end{align}
  The term
  \begin{displaymath}
    \iiiint\limits_{\Pi_T\times \Pi_T} \tilde V(s,y) \times [\eqref{eq:7}] \times \psi_h (t,x,s,y)  \d{x} \d{t} \d{y} \d{s}
  \end{displaymath}
  can be treated exactly in the same way, leading to
  \begin{equation}
    \label{eq:14}
    \!\!\!\!\iiiint\limits_{\Pi_T\times \Pi_T} \sgn\left(\rho(t,x)-\tilde \rho(t,x)\right) \tilde V(s,y) \, \partial_\rho f\left(t,x,\rho(t,x)\right)
    \partial_x \rho(t,x) \, \psi_h (t,x,s,y) \d{x} \d{t} \d{y} \d{s}.
  \end{equation}
  Introduce now the notation
  \begin{displaymath}
    \Upsilon(s,y) = \sgn\left(\rho(t,x)-\tilde \rho(t,x)\right) \tilde V(s,y) - \sgn\left(\rho(t,x)-\tilde \rho(s,y)\right) V(t,x)
  \end{displaymath}
  and apply \cite[Lemma 6.2]{ColomboRossiBordo}:
  \begin{align}
    \nonumber
    & \lim_{h \to 0} [\eqref{eq:14}] + [\eqref{eq:13}]
    \\ \label{eq:ok?}
    \!\!= \
    & \!\! \iint\limits_{\Pi_T}
      \sgn\left(\rho(t,x)-\tilde \rho(t,x)\right) \! \left(\tilde V(t,x)- V(t,x)\right)
      \partial_\rho f \left(t,x,\rho(t,x)\right) \, \partial_x \rho(t,x) \, \phi (t,x) d{x} \d{t}.
  \end{align}

  In order to deal with the last term, i.e.~\eqref{eq:7bis}, exploit
  the same regularisation of the sign function as above and compute
  \begin{align}
    \label{eq:8}
    & \iiiint\limits_{\Pi_T\times \Pi_T} \tilde V (s,y) \times [\eqref{eq:7bis}] \times \psi_h (t,x,s,y) \d{x}\d{t}\d{y}\d{s}
    \\ = \
    \label{eq:12}
    & \lim_{\alpha \to 0} \iiiint\limits_{\Pi_T\times \Pi_T}
      \bigl[
      s'_\alpha \left(\rho (t,x) -\tilde \rho (s,y)\right) \, \left(f\left(s,y,\rho (t,x)\right) - f \left(s,y,\tilde \rho (s,y)\right)\right)
    \\
    \label{eq:12b}
    & \qquad\qquad\quad
      - s'_\alpha\left(\rho (t,x) -\tilde \rho (t,x)\right) \, \left(f\left(t,x,\rho (t,x)\right) - f \left(t,x,\tilde \rho (t,x)\right)\right)
    \\
    \nonumber
    & \qquad\qquad\quad
      + s_\alpha \left(\rho (t,x) -\tilde \rho (s,y)\right) \partial_\rho f\left(s,y,\rho (t,x)\right)
    \\
    \nonumber
    & \qquad\qquad\quad
      -  s_\alpha\left(\rho (t,x) -\tilde \rho (t,x)\right) \partial_\rho f\left(t,x,\rho (t,x)\right)
      \bigr]
    \\
    \label{eq:12t}
    & \qquad \qquad
      \times    \tilde V (s,y) \, \partial_x \rho (t,x) \, \psi_h (t,x,s,y) \d{x}\d{t} \d{y} \d{s}.
  \end{align}
  By the Dominated Convergence Theorem, as $\alpha$ goes to $0$, we
  get
  \begin{align*}
    \iiiint\limits_{\Pi_T\times \Pi_T}[\eqref{eq:12}]\times[\eqref{eq:12t}]  \d{x}\d{t} \d{y} \d{s}\to \
    & 0,
    &
      \iiiint\limits_{\Pi_T\times \Pi_T} [\eqref{eq:12b}]\times[\eqref{eq:12t}] \d{x}\d{t} \d{y} \d{s} \to \
    & 0.
  \end{align*}
  Indeed,
  \begin{align*}
    & \modulo{\frac{2}{\alpha} \, Y\left(\frac{\rho - \tilde \rho}\alpha \right)
      \, \left(f\left(s,y,\rho (t,x)\right) - f \left(s,y,\tilde \rho (s,y)\right)\right)
      \tilde V (s,y) \, \partial_x \rho (t,x) \, \psi_h (t,x,s,y)
      }
    \\
    \leq \
    &  \frac{2}{\alpha} \, Y\left(\frac{\rho - \tilde \rho}\alpha \right)
      \int_{\tilde \rho}^\rho \modulo{\partial_\rho f (s,y,r)}\d{r}
      \norma{\tilde V}_{\L\infty (\Pi_T;\reali)} \modulo{\partial_x \rho (t,x)}
      \psi_h (t,x,s,y)
    \\
    \leq \
    & 2 \, \norma{Y}_{\L\infty (\reali;\reali)} \,
      \norma{\partial_\rho f }_{\L\infty (\Pi_T \times \reali;\reali)}
      \norma{\tilde V}_{\L\infty (\Pi_T;\reali)} \modulo{\partial_x \rho (t,x)}
      \psi_h (t,x,s,y) \quad \in \L1 (\Pi_T \times \Pi_T; \reali).
  \end{align*}
  Therefore we have
  \begin{align*}
    & [\eqref{eq:8}]
    \\
    =\
    & \iiiint\limits_{\Pi_T\times \Pi_T}
      \bigl[ \sgn \left(\rho (t,x) -\tilde \rho (s,y)\right) \partial_\rho f\left(s,y,\rho (t,x)\right)
      -  \sgn \left(\rho (t,x) -\tilde \rho (t,x)\right) \partial_\rho f\left(t,x,\rho (t,x)\right)
      \bigr]
    \\
    & \qquad \qquad
      \times    \tilde V (s,y) \, \partial_x \rho (t,x) \, \psi_h (t,x,s,y) \d{x}\d{t} \d{y} \d{s}.
  \end{align*}
  Introduce the notation
  $\Upsilon (s,y) = \sgn \left(\rho (t,x) -\tilde \rho
    (s,y)\right) \partial_\rho f\left(s,y,\rho (t,x)\right)$ and
  rewrite the equality above as follows
  \begin{align*}
    [\eqref{eq:8}]
    \leq \
    & \norma{\tilde V}_{\L\infty (\Pi_T;\reali)} \iiiint\limits_{\Pi_T\times \Pi_T}
      \modulo{\Upsilon (s,y) -\Upsilon (t,x)}
      \modulo{\partial_x \rho (t,x)} \, \psi_h (t,x,s,y) \d{x}\d{t} \d{y} \d{s},
  \end{align*}
  the left hand side clearly vanishing as $h$ goes to $0$, thanks
  to~\cite[Lemma~6.2]{ColomboRossiBordo} and to the fact that $\rho$
  has bounded variation.

  Collecting together all the estimates obtained in~\eqref{eq:h1},
  ~\eqref{eq:h2}, \eqref{eq:10ok}, \eqref{eq:11ok} and \eqref{eq:ok?},
  we get
  \begin{align}
    \!\!\!
    \nonumber
    &    \lim_{h \to 0} [\eqref{eq:partenza}]
    \\ \!\!\!
    \label{eq:arrivo1}
    = \
    & \!\!\iint\limits_{\Pi_T} \bigl\{
      \modulo{\rho (t,x) - \tilde \rho (t,x)} \partial_t \phi (t,x)
    \\ \!\!\!
    \nonumber
    & + \sgn \left(\rho (t,x) - \tilde\rho (t,x)\right) \tilde V (t,x)
      \left(f \left(t,x,\rho (t,x)\right) - f\left(t,x,\tilde\rho (t,x) \right)\right) \partial_x\phi (t,x)
    \\ \!\!\!
    \nonumber
    & + \sgn\left(\rho (t,x)  -\tilde \rho (t,x) \right)
      \left( \partial_x \tilde V (t,x)  - \partial_x V (t,x) \right)  f \left(t,x,\rho \right)
      \phi\! \left(t,x\right)
    \\\!\!\!
    \nonumber
    & +
      \sgn\left(\rho (t,x) - \tilde \rho (t,x)\right) \left(
      \tilde V (t,x) -  V (t,x) \right) \partial_x f \left(t,x, \rho (t,x)\right)
      \phi (t,x)
    \\ \!\!\!
    \label{eq:arrivo2}
    & + \sgn\left(\rho (t,x) - \tilde \rho (t,x)\right)\!
      \left(\tilde V (t,x) - V (t,x)\right)\!
      \partial_\rho f  \left(t,x,\rho (t,x) \right) \,
      \partial_x \rho (t,x) \, \phi\left(t,x\right)\bigr\}\!
      \d{x}\d{t}\!.
  \end{align}

  \smallskip

  Let now $h>0$ and $r>1$. Fix $0 < \tau < t <T$, define
  \begin{displaymath}
    \Phi_h (s) = \alpha_h (s - \tau) - \alpha_h (s-t),
    \qquad \mbox{where} \quad
    \alpha_h (z) = \int_{-\infty}^{z} Y_h (\zeta) \d\zeta,
  \end{displaymath}
  and
  \begin{displaymath}
    \Psi_r(x)=\int_{\R} Y (\modulo{x-y}) \, \caratt{\left\{\modulo{y}<r\right\}} (y) \d{y}.
  \end{displaymath}
  Observe that, as $h$ goes to $0$, $\Phi_h \to \caratt{[\tau,t]}$,
  and $\Phi'_h\to \delta_{\tau} - \delta_{t} $. Moreover,
  $\Psi'_r (x) = 0$ for $\modulo{x}< r -1$ or $\modulo{x}> r+1$ and,
  as $r$ tends to $+ \infty$, $\Psi_r \to \caratt{\R}$.  Choose
  $\phi(t,x) = \Phi_h (t) \, \Psi_r (x)$
  in~[\eqref{eq:arrivo1}$\cdots$\eqref{eq:arrivo2}] and pass to the
  limits $h \to 0$ and $r \to + \infty$ to obtain the desired
  estimate~[\eqref{eq:19}--\eqref{eq:20}]:
  \begin{align*}
    \int_{\R} \modulo{\rho (\tau,x) - \tilde \rho (\tau,x)} \d{x}
    -\int_{\R} \modulo{\rho (t,x) - \tilde \rho (t,x)} \d{x}
    &
    \\
    + \int_{\tau}^{t} \! \int_{\R}
    \big\{\modulo{\partial_x \tilde V (s,x)  - \partial_x V (s,x)} \, \modulo{f \left(s,x,\rho (s,x) \right)}
    &
    \\
    + \modulo{\tilde V (s,x) -  V (s,x)} \,\modulo{\partial_x f \left(s,x, \rho (s,x) \right)}
    &
    \\[2pt]
    + \modulo{\tilde V (s,x) - V (s,x)} \,
    \modulo{\partial_\rho f  \left(s,x,\rho (s,x) \right) } \,
    \modulo{\partial_x \rho (s,x) }\big\}
    & \d{x}\d{s} \geq \   0.
  \end{align*}
\end{proof}

\smallskip

\begin{proofof}{Theorem~\ref{thm:stabw}}
  We can apply Lemma~\ref{lem:dv} to problems~\eqref{eq:rho}
  and~\eqref{eq:rhot}. By Lemma~\ref{lem:Linf}, with obvious notation,
  for all $t \in [0,T]$ we have
  \begin{align*}
    \norma{\rho (t)}_{\L\infty (\R;\R)} \leq \
    &\norma{\rho_o}_{\L\infty (\R;\R)} \, e^{\mathcal{L} \, t}= M_t,
    &
      \norma{\tilde \rho (t)}_{\L\infty (\R;\R)} \leq \
    &\norma{\tilde \rho_o}_{\L\infty (\R;\R)} \, e^{\tilde {\mathcal{L}} \, t}=\tilde M_t.
  \end{align*}
  For the sake of simplicity introduce the space
  \begin{equation}
    \label{eq:sigma}
    \Sigma_t = [0,t] \times \R \times [0, \max\{M_t, \tilde M_t\}].
  \end{equation}
  Let $\tau \to 0$ in [\eqref{eq:19}$\cdots$\eqref{eq:20}]:
  \begin{align}
    \label{eq:4z}
    \norma{\rho(t)-\tilde \rho(t)}_{\L1(\R;\R)}
    \leq \
    & \norma{\rho_o-\tilde \rho_o}_{\L1(\R;\R)}
    \\
    \label{eq:4a}
    &  +
      \int_0^t\norma{f}_{\L\infty (\Sigma_s; \R)}  \int_{\R}
      \modulo{\partial_x V (s,x) -\partial_x \tilde V (s,x)} \d{x} \d{s}
    \\
    \label{eq:4c}
    & +  \int_{0}^{t}  \norma{\partial_x f}_{\L\infty (\Sigma_s; \R)} \int_{\R}
      \modulo{\tilde V (s,x)  - V (s,x)} \d{x} \d{s}
    \\
    \label{eq:4b}
    & +\int_{0}^{t} \norma{\partial_\rho f}_{\L\infty (\Sigma_s; \R)}   \int_{\R}
      \modulo{\partial_x \rho (s,x)} \, \modulo{V (s,x) -\tilde V (s,x)} \d{x}\d{s}  .
  \end{align}
  Consider~\eqref{eq:4a}. By the definitions of $V$ and $\tilde V$,
  compute
  \begin{align*}
    & \int_{\R} \modulo{\partial_x V (s,x) -\partial_x \tilde V (s,x)} \d{x}
    \\
    = \
    &
      \int_{\R}
      \modulo{v'((\rho (s) * w) (x)) \, \left(\rho (s) * \partial_x w\right)\! (x) -
      v'((\tilde\rho (s) * \tilde w) (x)) \, \left(\tilde \rho (s) * \partial_x \tilde w\right)\! (x)
      }\d{x}
    \\
    \leq \
    & \norma{v'}_{\L\infty (\R;\R)}
      \left(
      \norma{\rho(s)-\tilde \rho(s)}_{\L1(\R;\R)} \,
      \min\left\{ \norma{\partial_x w}_{\L1 (\R;\R)}, \norma{\partial_x \tilde w}_{\L1 (\R;\R)} \right\}
      \right.
    \\
    &\qquad\qquad\qquad \left.
      +
      \, \norma{\partial_x w - \partial_x \tilde w}_{\L1(\R;\R)} \,
      \min\left\{ \norma{\rho (s)}_{\L1 (\R;\R)}, \norma{\tilde \rho (s)}_{\L1 (\R;\R)} \right\}
      \right)
    \\
    &
      + \, \norma{v''}_{\L\infty (\R;\R)}
      \min\left\{  \norma{\rho (s)}_{\L1 (\R;\R)} \norma{\partial_x w}_{\L\infty (\R;\R)},
      \norma{\tilde \rho (s)}_{\L1 (\R;\R)}\norma{\partial_x \tilde w}_{\L\infty (\R;\R)} \right\}
    \\
    & \quad \times
      \left(
      \norma{\rho(s)-\tilde \rho(s)}_{\L1(\R;\R)} \,
      \min\left\{ \norma{w}_{\L1 (\R;\R)}, \norma{\tilde w}_{\L1 (\R;\R)} \right\}
      \right.
    \\
    & \qquad\qquad \left.
      +
      \norma{w - \tilde w}_{\L1(\R;\R)} \,
      \min\left\{ \norma{\rho (s)}_{\L1 (\R;\R)}, \norma{\tilde \rho (s)}_{\L1 (\R;\R)} \right\}
      \right)
    \\
    \leq \
    &
      \left(
      \norma{v'}_{\L\infty (\R;\R)}
      +
      \norma{v''}_{\L\infty (\R;\R)}
      \min\left\{  \norma{\rho_o}_{\L1 (\R;\R)} \norma{\partial_x w}_{\L\infty (\R;\R)},
      \norma{\tilde \rho_o}_{\L1 (\R;\R)}\norma{\partial_x \tilde w}_{\L\infty (\R;\R)} \right\}
      \right)
    \\
    & \times
      \left(
      \norma{\rho(s)-\tilde \rho(s)}_{\L1(\R;\R)}
      \min\left\{ \norma{w}_{\W11 (\R;\R)}, \norma{\tilde w}_{\W11 (\R;\R)} \right\} \,
      \right.
    \\
    & \qquad \left. +
      \norma{w - \tilde w}_{\W11(\R;\R)} \,
      \min\left\{ \norma{\rho_o}_{\L1 (\R;\R)}, \norma{\tilde \rho_o}_{\L1 (\R;\R)} \right\}
      \right),
  \end{align*}
  where we exploit also Lemma~\ref{lem:L1}.  Therefore,
  \begin{align}
    \nonumber
    &[\eqref{eq:4a}]
    \\
    \nonumber
    \leq \
    &
      \left(
      \norma{v'}_{\L\infty (\R;\R)}
      +
      \norma{v''}_{\L\infty (\R;\R)}
      \min\left\{ \norma{\rho_o}_{\L1 (\R;\R)}\norma{\partial_x w}_{\L\infty (\R;\R)},
      \norma{\tilde \rho_o}_{\L1 (\R;\R)} \norma{\partial_x \tilde w}_{\L\infty (\R;\R)} \right\}
      \right)
    \\
    \label{eq:4aok}
    & \times
      \left(
      \min\left\{ \norma{w}_{\W11 (\R;\R)}, \norma{\tilde w}_{\W11 (\R;\R)} \right\} \,
      \int_0^t  \norma{f}_{\L\infty (\Sigma_s;\R)}  \, \norma{\rho(s)-\tilde \rho(s)}_{\L1(\R;\R)} \d{s}
      \right.
    \\
    \nonumber
    & \qquad \left. +
      \, \norma{w - \tilde w}_{\W11(\R;\R)} \, \norma{f}_{\L\infty (\Sigma_t;\R)} \,
      \min\left\{ \norma{\rho_o}_{\L1 (\R;\R)}, \norma{\tilde \rho_o}_{\L1 (\R;\R)} \right\} \, t
      \right).
  \end{align}
  Consider~\eqref{eq:4c}: compute
  \begin{align*}
    \int_{\R}\modulo{V(s,x)-\tilde V(s,x)}\d{x}
    \leq \
    & \norma{v'}_{\L\infty (\reali;\reali)} \,
      \min\left\{\norma{w}_{\L1(\R;\R)},\norma{\tilde w}_{\L1(\R;\R)}\right\} \,
      \norma{\rho(s)-\tilde \rho(s)}_{\L1(\R;\R)}
    \\
    & + \norma{v'}_{\L\infty (\reali;\reali)}
      \min\left\{\norma{\rho(s)}_{\L1(\R;\R)},\norma{\tilde \rho(s)}_{\L1(\R;\R)}\right\}
      \norma{w-\tilde w}_{\L1(\R;\R)}
    \\
    \leq \
    &  \norma{v'}_{\L\infty (\reali;\reali)}\,
      \min\left\{\norma{w}_{\L1(\R;\R)},\norma{\tilde w}_{\L1(\R;\R)}\right\} \,
      \norma{\rho (s) - \tilde \rho (s)}_{\L1(\R;\R)}
    \\
    & + \norma{v'}_{\L\infty (\reali;\reali)}\,
      \min\left\{\norma{\rho_o}_{\L1(\R;\R)},\norma{\tilde \rho_o}_{\L1(\R;\R)}\right\}
      \norma{w-\tilde w}_{\L1(\R;\R)}.
  \end{align*}
  In this way we have
  \begin{align}
    \nonumber
    [\eqref{eq:4c}]\leq \
    &  \norma{v'}_{\L\infty (\reali;\reali)}
      \min\!\left\{\norma{w}_{\L1(\R;\R)},\norma{\tilde w}_{\L1(\R;\R)}\right\}
      \int_{0}^{t} \norma{\partial_x f}_{\L\infty(\Sigma_s;\R)} \, \norma{\rho(s)-\tilde \rho(s) }_{\L1(\R;\R)}\d{s}
    \\
    \label{eq:4cok}
    & + \norma{\partial_x f}_{\L\infty(\Sigma_t;\R)} \, \norma{v'}_{\L\infty (\reali;\reali)} \,
      \min\left\{\norma{\rho_o}_{\L1(\R;\R)},\norma{\tilde \rho_o}_{\L1(\R;\R)}\right\} \,
      \norma{w-\tilde w}_{\L1(\R;\R)}\, t .
  \end{align}
  Finally, consider~\eqref{eq:4b} and compute
  \begin{align*}
    \modulo{V (s,x) -\tilde V (s,x)}
    \leq \
    & \norma{v'}_{\L\infty(\R;\R)}\,
      \min\left\{\norma{w}_{\L\infty(\R;\R)}, \norma{\tilde w}_{\L\infty(\R;\R)} \right\}
      \norma{\rho(s)-\tilde \rho(s)}_{\L1(\R;\R)}
    \\
    &  + \norma{v'}_{\L\infty(\R;\R)}\,
      \min\left\{\norma{\rho(s)}_{\L\infty (\R;\R)},\norma{\tilde \rho(s)}_{\L\infty (\R;\R)}\right\} \,
      \norma{w-\tilde w}_{\L1(\R;\R)}
    \\
    \leq \
    & \norma{v'}_{\L\infty(\R;\R)}\,
      \min\left\{\norma{w}_{\L\infty(\R;\R)}, \norma{\tilde w}_{\L\infty(\R;\R)} \right\}
      \norma{\rho(s)-\tilde \rho(s)}_{\L1(\R;\R)}
    \\
    &  + \norma{v'}_{\L\infty(\R;\R)}\,
      \min\left\{M_s, \tilde M_s \right\} \,
      \norma{w-\tilde w}_{\L1(\R;\R)}.
  \end{align*}
  Hence,
  \begin{align}
    \nonumber
    [\eqref{eq:4b}] \leq \
    &
      \norma{v'}_{\L\infty(\R;\R)}\,
      \min\left\{\norma{w}_{\L\infty(\R;\R)}, \norma{\tilde w}_{\L\infty(\R;\R)} \right\}
    \\ \label{eq:4bok}
    & \times
      \int_0^t \norma{\partial_\rho f}_{\L\infty (\Sigma_s; \R)}
      \tv \left(\rho (s)\right) \norma{\rho(s)-\tilde \rho(s)}_{\L1(\R;\R)} \d{s}
    \\  \nonumber
    & + \norma{\partial_\rho f}_{\L\infty (\Sigma_t; \R)}
      \tv \left(\rho (t)\right) \norma{v'}_{\L\infty(\R;\R)}\,
      \min\left\{M_t, \tilde M_t \right\} \,
      \norma{w-\tilde w}_{\L1(\R;\R)} \, t .
  \end{align}
  Therefore, the inequality [\eqref{eq:4z}$\cdots$\eqref{eq:4b}] can
  be estimated as follows
  \begin{align*}
    & \norma{\rho(t)-\tilde \rho(t)}_{\L1(\R; \reali)}
    \\
    \leq \ &
    \norma{\rho_o-\tilde{\rho}_o}_{\L1(\R;\reali)}
    + a(t) \, \norma{w-\tilde w}_{\W11(\R;\R)}
    + \int_0^t b (s) \, \norma{\rho(s)-\tilde \rho(s)}_{\L1(\R;\reali)}\d{s},
  \end{align*}
  where, thanks to the total variation estimate provided by
  Proposition~\ref{prop:tv},
  \begin{align}
    \label{eq:a}
    &  a(t)
    \\
    \nonumber
    = \
    & t \,
      \left[\min\left\{ \norma{\rho_o}_{\L1 (\R;\R)}, \norma{\tilde \rho_o}_{\L1 (\R;\R)} \right\}
      \norma{f}_{\L\infty (\Sigma_t;\R)} \right.
    \\ \nonumber
    &
      \times \!\left(
      \norma{v'}_{\L\infty (\R;\R)}
      +
      \norma{v''}_{\L\infty (\R;\R)}
      \min\left\{ \norma{\rho_o}_{\L1 (\R;\R)}\norma{\partial_x w}_{\L\infty (\R;\R)},
      \norma{\tilde \rho_o}_{\L1 (\R;\R)} \norma{\partial_x \tilde w}_{\L\infty (\R;\R)} \right\}\!
      \right)
    \\ \nonumber
    &  + \min\left\{ \norma{\rho_o}_{\L1 (\R;\R)}, \norma{\tilde \rho_o}_{\L1 (\R;\R)} \right\}
      \norma{\partial_x f}_{\L\infty(\Sigma_t;\R)} \,
      \norma{v'}_{\L\infty (\R;\R)}
    \\ \nonumber
    & \left. + \norma{\partial_\rho f}_{\L\infty (\Sigma_t; \R)} \, \left(\mathcal K_2 t+\tv (\rho_o)\right)\,e^{\mathcal K_1 t}\,
      \norma{v'}_{\L\infty(\R;\R)}\, \min\left\{M_t, \tilde M_t \right\} \right]
  \end{align}
  and
  \begin{align}
    \label{eq:b}
    & b(s)
    \\ \nonumber
    = \
    & \norma{f}_{\L\infty (\Sigma_s;\R)}  \min\left\{ \norma{w}_{\W11 (\R;\R)}, \norma{\tilde w}_{\W11 (\R;\R)} \right\}
    \\  \nonumber
    & \times \!
      \left(
      \norma{v'}_{\L\infty (\R;\R)}
      +
      \norma{v''}_{\L\infty (\R;\R)}
      \min\left\{ \norma{\rho_o}_{\L1 (\R;\R)}\norma{\partial_x w}_{\L\infty (\R;\R)},
      \norma{\tilde \rho_o}_{\L1 (\R;\R)} \norma{\partial_x \tilde w}_{\L\infty (\R;\R)} \right\}\!
      \right)
    \\ \nonumber
    & + \norma{\partial_x f}_{\L\infty(\Sigma_s;\R)} \, \norma{v'}_{\L\infty (\reali;\reali)} \,
      \min\left\{\norma{w}_{\L1(\R;\R)},\norma{\tilde w}_{\L1(\R;\R)}\right\}
    \\ \nonumber
    & + \norma{\partial_\rho f}_{\L\infty (\Sigma_s; \R)} \,  \left(\mathcal K_2 s+\tv (\rho_o)\right)\,e^{\mathcal K_1 s} \,
      \norma{v'}_{\L\infty(\R;\R)} \,
      \min\left\{\norma{w}_{\L\infty(\R;\R)}, \norma{\tilde w}_{\L\infty(\R;\R)} \right\},
  \end{align}
  $\mathcal{K}_1$ and $\mathcal{K}_2$ being specified in~\eqref{eq:2}.
  An application of Gronwall Lemma yields
  \begin{align*}
    \norma{\rho(t)-\tilde \rho(t)}_{\L1(\R; \reali)}
    \leq \
    & \norma{\rho_o-\tilde{\rho}_o}_{\L1(\R;\reali)}
      + a(t) \, \norma{w-\tilde w}_{\W11(\R;\R)}
    \\
    & + \int_0^t
      \left(\norma{\rho_o-\tilde{\rho}_o}_{\L1(\R;\reali)} + a(s)  \right)
      b (s) \, \exp\left(\int_s^t b (r) \d{r} \right)\d{s}.
  \end{align*}
  Since $a (s) \leq a (t)$ for any $s \in [0,t]$ and
  \begin{displaymath}
    \int_0^t  b (s) \, \exp\left(\int_s^t b (r) \d{r} \right)\d{s}
    =  \left[ - \exp \left(\int_s^t b (r) \d{r}\right) \right]_0^t
    = -1 +  \exp\left(\int_0^t b (r) \d{r} \right),
  \end{displaymath}
  we obtain
  \begin{equation}
    \label{eq:15}
    \norma{\rho(t)-\tilde \rho(t)}_{\L1(\R; \reali)}
    \leq
    \left( \norma{\rho_o-\tilde{\rho}_o}_{\L1(\R;\reali)} + a(t) \, \norma{w-\tilde w}_{\W11(\R;\R)}  \right) \exp\left(\int_0^t b (r) \d{r} \right),
  \end{equation}
  concluding the proof.
\end{proofof}

\begin{remark}
  {\rm Notice that, when $t=0$, the right hand side of~\eqref{eq:15}
    is equal to $\norma{\rho_o-\tilde{\rho}_o}_{\L1(\R;\reali)}$,
    since $a (0) = 0$.}
\end{remark}

\begin{remark}\label{rem:Magali}
  {\rm Compare our estimate~\eqref{eq:15} with the one
    in~\cite[Theorem~4.1]{Betancourt2011}:
    \begin{displaymath}
      \norma{\rho (t) - \tilde \rho (t)}_{\L1 (\reali;\reali)} \leq e^{C_2 \, t} \norma{\rho_o - \tilde \rho}_{\L1 (\reali;\reali)},
    \end{displaymath}
    where
    \begin{align*}
      C_2 = \
      \norma{f}_{\L\infty(\Sigma_t;\R)}
      & \left(
        \norma{v'}_{\L\infty(\R;\R) } \norma{\partial_x w}_{\L1(\R;\R)} \right.\\
      &\left. +
        \norma{v''}_{\L\infty(\R;\R) } \norma{\partial_x w}_{\L\infty(\R;\R)} \norma{w}_{\L1(\R;\R)}
        \min\left\{ \norma{\rho_o}_{\L1(\R;\R)},  \norma{\tilde \rho_o}_{\L1(\R;\R)}  \right\}
        \right)
      \\
      + \norma{f'}_{\L\infty(\Sigma_t;\R)}
      & \, \tv \left(\rho (t)\right) \,  \norma{v'}_{\L\infty(\R;\R)} \,  \norma{w}_{\L\infty(\R;\R)}.
    \end{align*}
    The main hypotheses there are the following:
    \begin{itemize}
    \item $f (t,x,\rho) = f (\rho)$;
    \item $w = \tilde w$, thus the kernel functions are the same;
    \item different initial data: $\rho_o \neq \tilde \rho_o$.
    \end{itemize}
    It is immediate to see that, once the estimate for the total
    variation of $\rho (t)$ is inserted, the bound $C_2$ bears a
    strong resemblance with our $b (t)$~\eqref{eq:b}, provided the
    $\L1$-norm of the kernel $w$ and of its derivative are controlled
    by $\norma{w}_{\W11(\R;\R)}$. }
\end{remark}

\begin{remark}
  {\rm One may wonder why there is the need to exploit the doubling of
    variables method and to go through all the steps of the proof
    instead of using the ready-made estimate provided
    in~\cite[Theorem~2.5 or Proposition~2.9]{MercierV2}.  The reason
    lies in the coefficient $\kappa^*$ appearing in the estimates
    presented in that work. Indeed, with our notation, this
    coefficient reads
    \begin{displaymath}
      \kappa ^* = \norma{\partial_\rho \partial_x \left(f (t,x,\rho) \left(V (t,x)- \tilde V (t,x)\right)\right)}_{\L\infty (\Sigma_t;\reali)}.
    \end{displaymath}
    Computing the derivatives yields
    \begin{displaymath}
      \kappa^* \leq
      \norma{\partial_x \partial_\rho f}_{\L\infty (\Sigma_t;\reali)} \norma{V -\tilde V}_{\L\infty ([0,t] \times \reali;\reali)}
      + \norma{\partial_\rho f}_{\L\infty (\Sigma_t;\reali)} \norma{\partial_x V - \partial_x \tilde V}_{\L\infty ([0,t] \times \reali;\reali)}.
    \end{displaymath}
    Substitute now the definitions of $V$ and $\tilde V$, using also
    the estimates for~\eqref{eq:4a} and~\eqref{eq:4b} computed in the
    proof of Theorem~\ref{thm:stabw}: we obtain an estimate for
    $\kappa^*$ depending on the term
    $\norma{\rho-\tilde \rho}_{\L\infty ([0,t];\L1(\reali;\reali))}$.
    Going back to the estimate presented in~\cite{MercierV2},
    we see that the coefficient $\kappa^*$ appears in the term
    $e^{\kappa^* \, t} \norma{\rho_o -\tilde \rho_o}_{\L1
      (\reali;\reali)}$. Therefore, since the final goal is to control
    from above
    $\norma{\rho (t) -\tilde \rho (t)}_{\L1 (\reali;\reali)}$, we get
    an implicit estimate for it, which is clearly not what desired. }
\end{remark}

\begin{proofof}{Theorem~\ref{thm:stabv}}
  We can apply Lemma~\ref{lem:dv} to problems~\eqref{eq:rho1}
  and~\eqref{eq:rhot1}.  Let us start from the
  inequality~[\eqref{eq:19}--\eqref{eq:20}].  Introduce the following
  notation, based on Lemma~\ref{lem:Linf}:
  \begin{align*}
    \norma{\rho (t)}_{\L\infty (\R;\R)} \leq \
    &\norma{\rho_o}_{\L\infty (\R;\R)} \, e^{\mathcal{L} \, t}
    &
      \norma{\tilde \rho (t)}_{\L\infty (\R;\R)} \leq \
    &\norma{\rho_o}_{\L\infty (\R;\R)} \, e^{\tilde {\mathcal{L}} \, t}.
  \end{align*}
  Define
  $\mathcal{G}_t = \norma{\rho_o}_{\L\infty (\R;\R)} \,
  e^{\max\{\mathcal{L}, \tilde {\mathcal{L}}\} \, t}$. Similarly
  to~\eqref{eq:sigma}, introduce the space
  \begin{displaymath}
    \Sigma_t = [0,t] \times \reali \times [0, \mathcal{G}_t].
  \end{displaymath}
  Let $\tau \to 0$ in~[\eqref{eq:19}--\eqref{eq:20}] and recall also
  the assumption
  $\sup_{t,x}\modulo{\partial_x f(t,x,\rho)} < C \modulo{\rho}$:
  \begin{align}
    \label{eq:17a}
    \norma{\rho(t)-\tilde \rho(t)}_{\L1(\R;\R)}
    \leq \
    &
      \int_0^t\norma{f}_{\L\infty (\Sigma_s; \R)}  \int_{\R}
      \modulo{\partial_x V (s,x) -\partial_x \tilde V (s,x)} \d{x} \d{s}
    \\
    \label{eq:17b}
    & +  \int_{0}^{t}  \int_{\R} C \modulo{\rho (t,x)} \,
      \modulo{\tilde V (s,x)  - V (s,x)} \d{x} \d{s}
    \\
    \label{eq:17c}
    & +\int_{0}^{t} \norma{\partial_\rho f}_{\L\infty (\Sigma_s; \R)}   \int_{\R}
      \modulo{\partial_x \rho (s,x)} \, \modulo{V (s,x) -\tilde V (s,x)} \d{x}\d{s}  .
  \end{align}
  By the definitions of $V$ and $\tilde V$, compute:
  \begin{align*}
    &\modulo{V (s,x) - \tilde V (s,x)}
    \\
    = \
    & \modulo{v ( (\rho (s) * w ) (x)) - \tilde v ( (\tilde \rho (s) * w ) (x))}
    \\
    \leq \
    & \min\left\{\norma{v'}_{\L\infty (\reali;\reali)}, \norma{\tilde v'}_{\L\infty (\reali;\reali)}\right\}
      \norma{w}_{\L\infty (\reali;\reali)} \norma{\rho (s) - \tilde \rho (s)}_{\L1 (\reali;\reali)}
      +
      \norma{v - \tilde v}_{\L\infty (\reali;\reali)}
  \end{align*}
  and
  \begin{align*}
    & \int_\reali \modulo{\partial_x V (s,x) -\partial_x \tilde V (s,x)} \d{x}
    \\
    =\
    & \int_{\R}
      \modulo{v'((\rho (s) * w) (x)) \, \left(\rho (s) * \partial_x w\right) (x) -
      \tilde v'((\tilde\rho (s) * w) (x)) \, \left(\tilde \rho (s) * \partial_x  w\right) (x)
      }\d{x}
    \\
    \leq \
    & \min\left\{\norma{v'}_{\L\infty (\reali;\reali)}, \norma{\tilde v'}_{\L\infty (\reali;\reali)}\right\}
      \norma{\partial_x w}_{\L1 (\reali;\reali)} \norma{\rho (s) - \tilde \rho (s)}_{\L1 (\reali;\reali)}
    \\
    & +
      \norma{v' - \tilde v'}_{\L\infty (\reali;\reali)} \norma{\partial_x w}_{\L1 (\R;\R)}
      \min\left\{\norma{\rho (s)}_{\L1 (\reali;\reali)}, \norma{\tilde \rho (s)}_{\L1 (\reali;\reali)}\right\}
    \\
    \leq \
    &  \min\left\{\norma{v'}_{\L\infty (\reali;\reali)}, \norma{\tilde v'}_{\L\infty (\reali;\reali)}\right\}
      \norma{\partial_x w}_{\L1 (\reali;\reali)} \norma{\rho (s) - \tilde \rho (s)}_{\L1 (\reali;\reali)}
    \\
    & +
      \norma{v' - \tilde v'}_{\L\infty (\reali;\reali)} \norma{\partial_x w}_{\L1 (\R;\R)}
      \norma{\rho_o}_{\L1 (\reali;\reali)},
  \end{align*}
  where we exploit also Lemma~\ref{lem:L1}.  Therefore the
  inequality~[\eqref{eq:17a}--~\eqref{eq:17c}] can be estimated as
  follows:
  \begin{align*}
    &\norma{\rho (t) - \tilde \rho (t)}_{\L1 (\reali;\reali)}
    \\ \leq \
    &  c_1 (t) \, \norma{v - \tilde v}_{\L\infty (\reali;\reali)}
      +
      c_2 (t) \, \norma{v' - \tilde v'}_{\L\infty (\reali;\reali)}
      +
      \int_0^t c_3 (s) \, \norma{\rho (s) - \tilde \rho (s)}_{\L1 (\reali;\reali)} \d{s},
  \end{align*}
  where, thanks to the total variation estimate provided by
  Proposition~\ref{prop:tv},
  \begin{align}
    \label{eq:c1}
    c_1 (t) = \
    & t \, \left(C \, \norma{\rho_o}_{\L1 (\reali;\reali)}
      + (\mathcal{K}_2 \, t + \tv (\rho_o)) \, e^{\mathcal{K}_1 \, t} \, \norma{\partial_\rho f}_{\L\infty (\Sigma_t;\reali)} \right),
    \\
    \label{eq:c2}
    c_2 (t) = \
    & t \, \norma{f}_{\L\infty (\Sigma_t;\reali)} \, \norma{\partial_x w}_{\L1 (\reali;\reali)} \, \norma{\rho_o}_{\L1 (\reali;\reali)},
    \\
    \nonumber
    c_3 (s) = \
    & \norma{f}_{\L\infty (\Sigma_s;\reali)}
      \min\left\{\norma{v'}_{\L\infty (\reali;\reali)}, \norma{\tilde v'}_{\L\infty (\reali;\reali)}\right\}
      \norma{\partial_x w}_{\L1 (\reali;\reali)}
    \\
    \label{eq:c3}
    & +
      \left(C \, \norma{\rho_o}_{\L1 (\reali;\reali)}
      + (\mathcal{K}_2 \, s + \tv (\rho_o)) \, e^{\mathcal{K}_1 \, s} \, \norma{\partial_\rho f}_{\L\infty (\Sigma_s;\reali)} \right)
    \\
    \nonumber
    & \quad\times \min\left\{\norma{v'}_{\L\infty (\reali;\reali)}, \norma{\tilde v'}_{\L\infty (\reali;\reali)}\right\}
      \norma{w}_{\L\infty (\reali;\reali)},
  \end{align}
  $\mathcal{K}_1$ and $\mathcal{K}_2$ being specified
  in~\eqref{eq:2}. An application of Gronwall Lemma yields
  \begin{align*}
    \norma{\rho (t) - \tilde \rho (t)}_{\L1 (\reali;\reali)} \leq \
    & \left(
      c_1 (t) \, \norma{v - \tilde v}_{\L\infty (\reali;\reali)}
      +
      c_2 (t) \, \norma{v' - \tilde v'}_{\L\infty (\reali;\reali)}
      \right) \, \exp\left(\int_0^t c_3 (s) \d{s} \right),
  \end{align*}
  concluding the proof.
\end{proofof}

\section{Numerical Integrations}
\label{sec:numint}
In this section, we investigate the dependence of solutions to
\eqref{eq:1} on the kernel and the velocity function via numerical
integrations.  To this end, we discretize~\eqref{eq:1} on a fixed grid
given by the cells interfaces $x_{j+\frac{1}{2}}=j \Delta x$ and the cells
centres $x_j=(j-\frac{1}{2})\Delta x$ for $j\in \interi,$ taking a
space step $\Delta x$ and a time step $\Delta t$, so that
$t^n=n\Delta t$ is the time mesh.
The Lax-Friedrichs flux adapted to~\eqref{eq:1} is given by
\begin{equation} \label{eq:LFflux} F ^n_{j+1/2}:= \frac12\left( f
    (t^n, x_{j},\rho^n_j) v (R^n_j) + f (t^n, x_{j+1},\rho^n_{j+1}) v
    (R^n_{j+1}) \right) -\frac{\alpha}{2} (\rho^n_{j+1}-\rho^n_j)
\end{equation}
where $\alpha\geq0$ is the viscosity coefficient and
$R^n_j := \Delta x \displaystyle{\sum_{k \in \Z}}\rho^n_{j+k}
w^k_\eta$, denoting $w^k_\eta:= w_\eta(k \Delta x)$ for $k\in \Z$.  In
this way we have the finite volume scheme
\begin{equation}
  \label{eq:18}
  \rho^{n+1}_j =
  \rho^n_j -   \lambda
  \left[
    F ^n_{j+1/2}
    -
    F ^n_{j-1/2}
  \right],
\end{equation}
with $\lambda= \Delta t/ \Delta x.$ A rigorous study of the
convergence of Lax-Friedrichs type schemes for non-local conservation
laws has been carried out in \cite{ACG2015, AmorimColomboTexeira,
  BlandinGoatin2016}.  Here we limit the study to the derivation of
sufficient conditions ensuring that the above discretization~\eqref{eq:LFflux}--\eqref{eq:18} is positivity preserving.

\begin{lemma} \label{lem:CFL} For any $T>0$, under the CFL conditions
  \begin{equation}
    \label{eq:CFL}
    \lambda
    \left( \alpha +
      \left(C \, \Delta x + 2 \, \norma{\partial_\rho f}_{\L\infty(\Sigma_T;\R)}\right)\norma{v}_{\L\infty(\R;\R)}\right) < 1,
  \end{equation}
  \begin{equation}
    \label{eq:alpha}
    \alpha \geq \norma{\partial_\rho f}_{\L\infty(\Sigma_T;\R)}\norma{v}_{\L\infty(\R;\R)},
  \end{equation}
  the scheme~\eqref{eq:LFflux}--\eqref{eq:18} is positivity preserving
  on $[0,T]\times\R$.
\end{lemma}

\begin{proof}
  Let us assume that $\rho_j^n \geq 0$ for all $j \in \interi$.  It
  suffices to prove that $\rho^{n+1}_j$ in~\eqref{eq:18} is
  non-negative.  For the sake of simplicity, in the following we omit
  the dependence on $n$ and introduce the notation
  $f_i (\rho_j) = f(t^n, x_i, \rho_j)$ and $v_j = v (R^n_j)$.  Compute
  \begin{align*}
    \rho^{n+1}_j = \
    & \rho_j + \frac{\lambda \, \alpha}{2} (\rho_{j+1} - 2 \, \rho_j + \rho_{j-1})
      - \frac{\lambda}{2} \left[f_{j+1} (\rho_{j+1}) \, v_{j+1} - f_{j-1} (\rho_{j-1}) \, v_{j-1}\right]
    \\
    = \
    & \rho_j (1- \lambda \, \alpha) + \frac{\lambda \, \alpha}{2} (\rho_{j+1} + \rho_{j-1})
    \\
    & -\frac{\lambda}{2}\left[
      \left(f_{j+1} (\rho_{j+1}) - f_{j+1} (\rho_j) \right) v_{j+1}
      +
      \left(f_{j-1} (\rho_{j}) - f_{j-1} (\rho_{j-1}) \right) v_{j-1}\right.
    \\
    & \qquad\,\, \left.
      +
      \left(f_{j+1} (\rho_{j}) - f_{j-1} (\rho_j) \right) v_{j+1}
      +
      f_{j-1} (\rho_j)\left(v_{j+1} -v_{j-1}\right)
      \right]
    \\
    = \
    & \rho_j\left(
      1  - \lambda \, \alpha
      + \frac{\lambda}{2} \,
      \frac{f_{j+1} (\rho_{j+1}) - f_{j+1} (\rho_j)}{\rho_{j+1}- \rho_j} \, v_{j+1}
      -  \frac{\lambda}{2} \,
      \frac{f_{j-1} (\rho_{j}) - f_{j-1} (\rho_{j-1})}{\rho_j - \rho_{j-1}} \, v_{j-1}
      \right)
    \\
    & + \rho_{j+1}\left(
      \frac{\lambda \, \alpha}{2}
      - \frac{\lambda}{2} \,
      \frac{f_{j+1} (\rho_{j+1}) - f_{j+1} (\rho_j)}{\rho_{j+1}- \rho_j} \, v_{j+1}
      \right)
    \\
    & + \rho_{j-1}\left(
      \frac{\lambda \, \alpha}{2}
      + \frac{\lambda}{2} \,
      \frac{f_{j-1} (\rho_{j}) - f_{j-1} (\rho_{j-1})}{\rho_j - \rho_{j-1}} \,v_{j-1}
      \right)
    \\
    & -\frac{\lambda}{2} \, v_{j+1} \left(f_{j+1} (\rho_{j}) - f_{j-1} (\rho_j) \right)
      - \frac{\lambda}{2}  \, f_{j-1} (\rho_j) \left(v_{j+1} -v_{j-1}\right).
  \end{align*}
  Observe that, thanks to the assumption~\eqref{eq:alpha} on $\alpha$,
  \begin{align*}
    \alpha +  \frac{f_{j-1} (\rho_{j}) - f_{j-1} (\rho_{j-1})}{\rho_j - \rho_{j-1}} \, v_{j-1} \! = \
    & \!\alpha + \partial_\rho f_{j-1} (\zeta_{j-1/2}) \, v_{j-1}
      \!\geq\! \alpha - \!\norma{\partial_\rho f}_{\L\infty(\Sigma_T;\R)} \norma{v}_{\L\infty(\R;\R)} \!\geq\! 0,
    \\
    \alpha -  \frac{f_{j+1} (\rho_{j+1}) - f_{j+1} (\rho_j)}{\rho_{j+1}- \rho_j} \, v_{j+1} \! = \
    & \!\alpha - \partial_\rho f_{j+1} (\zeta_{j+1/2}) \, v_{j+1}
      \!\geq\! \alpha - \!\norma{\partial_\rho f}_{\L\infty(\Sigma_T;\R)} \norma{v}_{\L\infty(\R;\R)} \!\geq\! 0.
  \end{align*}
  Moreover,
  \begin{displaymath}
    v_{j+1} \left(f_{j+1} (\rho_{j}) - f_{j-1} (\rho_j) \right)
    \leq 2 \, C \, \norma{v}_{\L\infty(\R;\R)} \, \Delta x \, \rho_j
  \end{displaymath}
  and
  \begin{displaymath}
    f_{j-1} (\rho_j) \left(v_{j+1} -v_{j-1}\right)
    \leq 2 \, \norma{\partial_\rho f}_{\L\infty(\Sigma_T;\R)} \norma{v}_{\L\infty(\R;\R)}\rho_j.
  \end{displaymath}
  Hence,
  \begin{align*}
    \rho_j  & \left(
              1  - \lambda \, \alpha
              + \frac{\lambda}{2} \,
              \frac{f_{j+1} (\rho_{j+1}) - f_{j+1} (\rho_j)}{\rho_{j+1}- \rho_j} \, v_{j+1}
              -  \frac{\lambda}{2} \,
              \frac{f_{j-1} (\rho_{j}) - f_{j-1} (\rho_{j-1})}{\rho_j - \rho_{j-1}} \, v_{j-1}
              \right)
    \\
            & -\frac{\lambda}{2} \, v_{j+1} \left(f_{j+1} (\rho_{j}) - f_{j-1} (\rho_j) \right)
              - \frac{\lambda}{2}  \, f_{j-1} (\rho_j) \left(v_{j+1} -v_{j-1}\right)
    \\
    \geq \
            & \rho_j \left(
              1  - \lambda \, \alpha
              -2 \, \lambda \, \norma{\partial_\rho f}_{\L\infty(\Sigma_T;\R)} \norma{v}_{\L\infty(\R;\R)}
              -\lambda \,  C   \, \norma{v}_{\L\infty(\R;\R)} \, \Delta x
              \right) \geq 0,
  \end{align*}
  by the CFL condition~\eqref{eq:CFL}.
\end{proof}

\medskip
Fix $T=0.5$. Let us now consider the following problem:
\begin{equation} \label{eq:sample}
  \begin{cases}
    \partial_t \rho+\partial_x(f(t,x,\rho)v(w_{\eta,\delta}\ast\rho))=0, & t\in[0,T],~x\in \ ]-1,1[ \ , \\
    \rho(0,x) = 0.6,
  \end{cases}
\end{equation}
with periodic boundary conditions at $x=\pm 1$ and
\begin{align}
  \label{eq:f}
  &f(t,x,\rho)=V_{\max}(t,x) \rho (1-\rho), & \rho\in [0,1],\\
  \label{eq:vel}
  &v(\rho)=(1-\rho)^{m-1} (1+\rho)^m, & m\in\N,\\
  \label{eq:kernel}
  & w_{\eta,\delta} (x)=\frac{1}{\eta^{6}}\frac{16}{5 \pi}(\eta^2-(x-\delta)^2)^\frac{5}{2} \chi_{[-\eta+\delta,\eta+\delta]},  & \eta\in \ ]0,1],~\delta\in [-\eta,\eta].
\end{align}
In~\eqref{eq:f}, $V_{\max}(t,x)$ is given by the convolution between
the gaussian kernel
$g(x)=\frac{1}{\sigma \sqrt{2
    \pi}}e^{-\frac{1}{2}(\frac{x}{\sigma})^2}$ with $\sigma=10$ and
the following piece-wise constant function:
\begin{equation*}
  \phi(t,x)=\begin{cases}
    7 & \text{ if } x\in \ ]-1, -1/3] \ \cup \ ]1/3,  1],\\
    3 & \text{ if } x\in \ ]-1/3,1/3],~t\in [0, 1/6] \ \cup \ ]1/3,1/2],\\
    1.5 & \text{ if } x\in \ ]-1/3,1/3],~t\in ]1/6, 1/3] ,
  \end{cases}
\end{equation*}
see Figure~\ref{fig:vmax}. In~\eqref{eq:kernel}, the parameter $\eta$
represents the radius of the support of the kernel function
$w_{\eta,\delta}$, while $\delta$ is the point at which the maximum is
attained.
\begin{figure}[!h]
  \centering
  \includegraphics[width=0.55\textwidth, trim=15 5 35 20, clip=true]{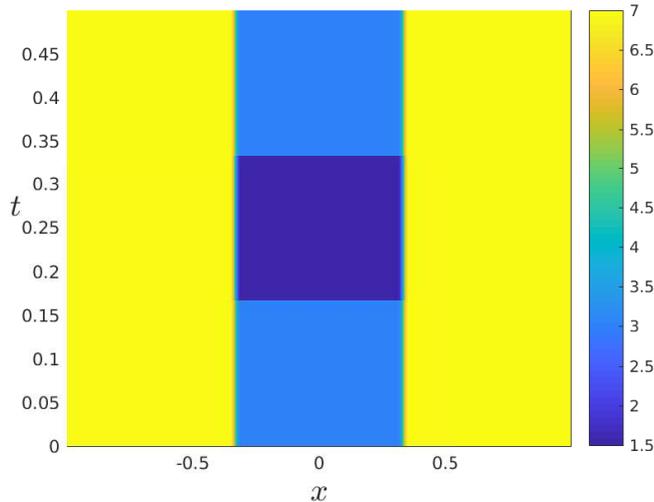}
  \caption{2D plot of the function $V_{max}(t,x)$}
  \label{fig:vmax}
\end{figure}

The above equations~\eqref{eq:sample}--\eqref{eq:kernel} describe the
traffic flow on a circular road with variable speed limit in space and
time, starting from a constant initial density $\rho_o \equiv 0.6$
(for simplicity, the maximal density is here normalised to 1).

As a metric of traffic congestion, we consider the two following
functionals~\cite{ColomboGoatinRosini2011, ColomboGroli2004,
  ColomboRossi2017}:
\begin{align}
  \label{J}
  J(T)&=\int_{0}^{T} \d{} \modulo{\del_x \rho(t,\cdot) } \d t ,
  \\
  \label{functPsi}
  \Psi(T;a,b)&=\int_{0}^{T}\int_{a}^{b}\phi(\rho(t,x)) \,\d x\,\d t,
\end{align}
where
\begin{equation*}
  \phi(r)=\begin{cases}
    0  &r <0.75,\\
    10\, {r}-7.5  &0.75 \leq r \leq 0.85,\\
    1 &0.85 <r\leq 1.
  \end{cases}
\end{equation*}
The functional $J$ defined in~\eqref{J} measures the integral with
respect to time of the spatial total variation of the traffic
density. The results of Theorems~\ref{thm:stabw} and~\ref{thm:stabv}
apply to the present setting and ensure the continuous dependence of
$J$ on the parameters $m$, $\eta$ and $\delta$. Indeed, the map
$\delta \to w_{\eta,\delta}$ is Lipschitz continuous with respect to
the $\W11$ distance, the map $\eta \to w_{\eta,\delta}$ is continuous
with respect to the $\W11$ distance and the map $m \to v$ is
continuous with respect to the $\W1\infty$
distance. Theorem~\ref{thm:stabw} then ensures that the map
$w_{\eta,\delta} \to \rho$, where $\rho$
solves~\eqref{eq:sample}--\eqref{eq:kernel}, is continuous with
respect to the $\W11$ distance, while Theorem~\ref{thm:stabv} ensures
the continuity of the map $v \to \rho$. Finally, the map $\rho \to J$
is lower semicontinuous, as showed in~\cite[Lemma
2.1]{ColomboGroli2004}.  Therefore, any minimising sequence of solutions
converges, guaranteeing the existence of optimal choices of the parameters
$\eta$, $\delta$ and $m$.

The functional $\Psi$ in~\eqref{functPsi} was introduced in
\cite{ColomboRossi2017} and it is obviously continuous with respect to
$\rho$ in the $\L1$-distance.  It measures the queue of the solution
in the space interval $[a,b]$, which is chosen equal to $[-4/5,-1/3]$
in the numerical simulations below.

For the tests, we fix the space discretization mesh to
$\Delta x=0.001.$ Figures~\ref{fig:eta}--\ref{fig:delta} show the
values of the functionals $J$ and $\psi$ when we vary the value of one
of the parameters $\eta$, $\delta$ and $m$, keeping the other
fixed. In particular, the functionals are evaluated on the following
grids:
\[
  \eta = 0.1 \colon 0.1 \colon 1, \qquad \delta = -0.1 \colon 0.02
  \colon 0.1,\qquad m = 1 \colon 1 \colon 10.
\]
We observe that the functionals are in general not monotone and
display some extrema in the considered intervals.
Figures~\ref{fig:etaEx}, \ref{fig:deltaEx} and~\ref{fig:mEx} show the
behaviour of the solutions corresponding to some of these extremal
values.  More precisely, Figures~\ref{fig:eta02m3}, \ref{fig:eta05m3}
and \ref{fig:eta1} show the solutions corresponding to
$\eta=0.2,~0.5,~1$ for $m=3$ and centered kernel ($\delta=0$). In
particular, the solutions displayed in~\ref{fig:eta02m3}
and~\ref{fig:eta1} correspond to the minimum and maximum values of the
functional $J$~\eqref{J} for $\eta\in [0.1,1]$ (see
Figure~\ref{fig:eta}, left).  Figure~\ref{fig:delta04} shows the
solution obtained for $\delta=-0.04$ (and $m=3$, $\eta=0.1$) and
corresponding to the point of minimum of both $J$ and $\Psi$
functionals, while Figures~\ref{fig:delta06} and~\ref{fig:delta10}
correspond to the points of maximum of the functionals $J$ and $\Psi$,
respectively (see Figure~\ref{fig:delta}).  Finally, in
Figures~\ref{fig:eta01m3} and~\ref{fig:eta01m10} we give the solutions
corresponding to the maximum and minimum points of the functional $J$
for $m\in \{1,\ldots, 10\}$ for $\eta=0.1$ and $\delta=0$ (see
Figure~\ref{fig:m}).

\begin{figure}[!h]
  \centering
  \includegraphics[width=0.5\textwidth, trim = 15 5 35 20, clip=true]{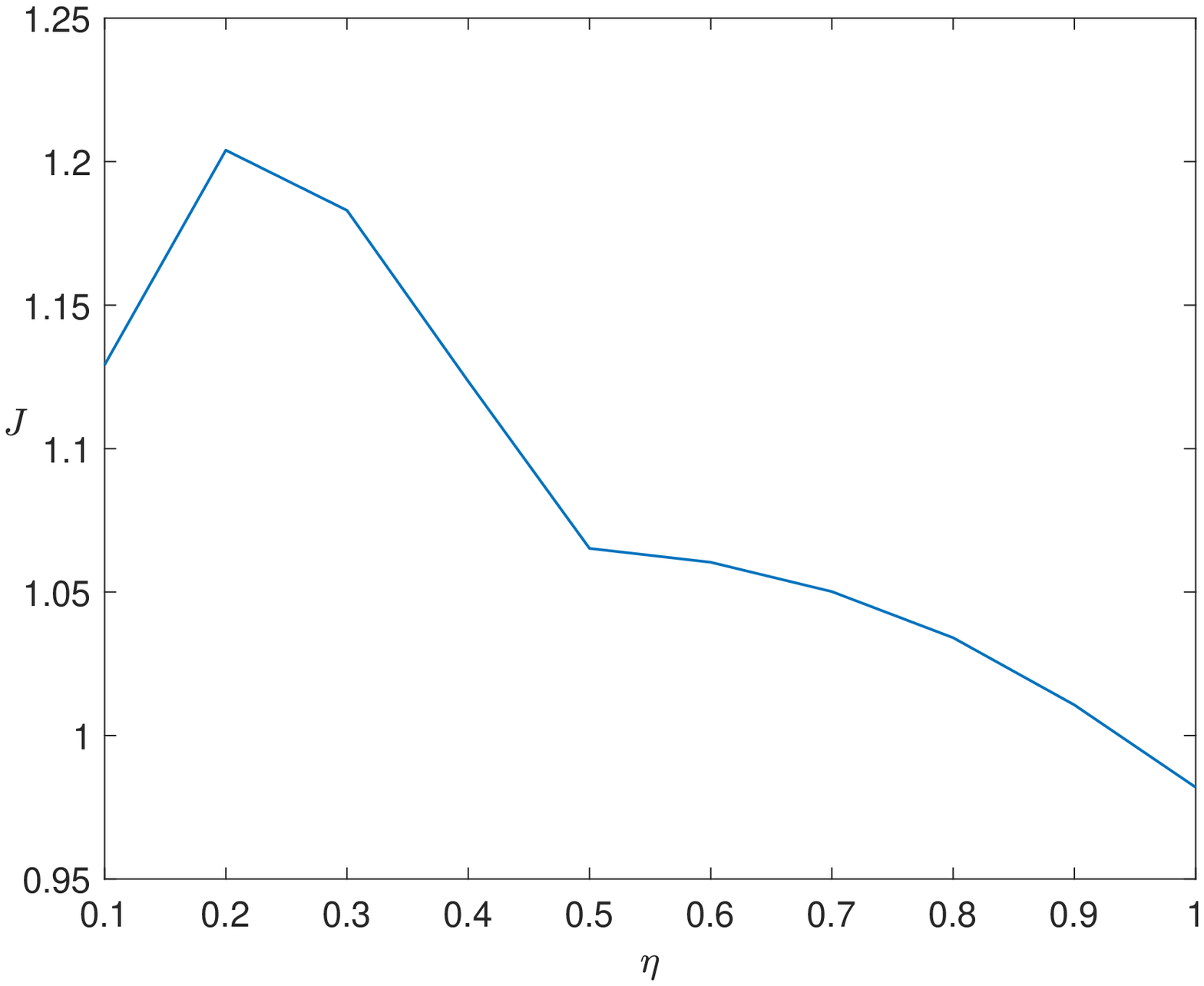}%
  \includegraphics[width=0.5\textwidth, trim = 15 5 35 20, clip=true]{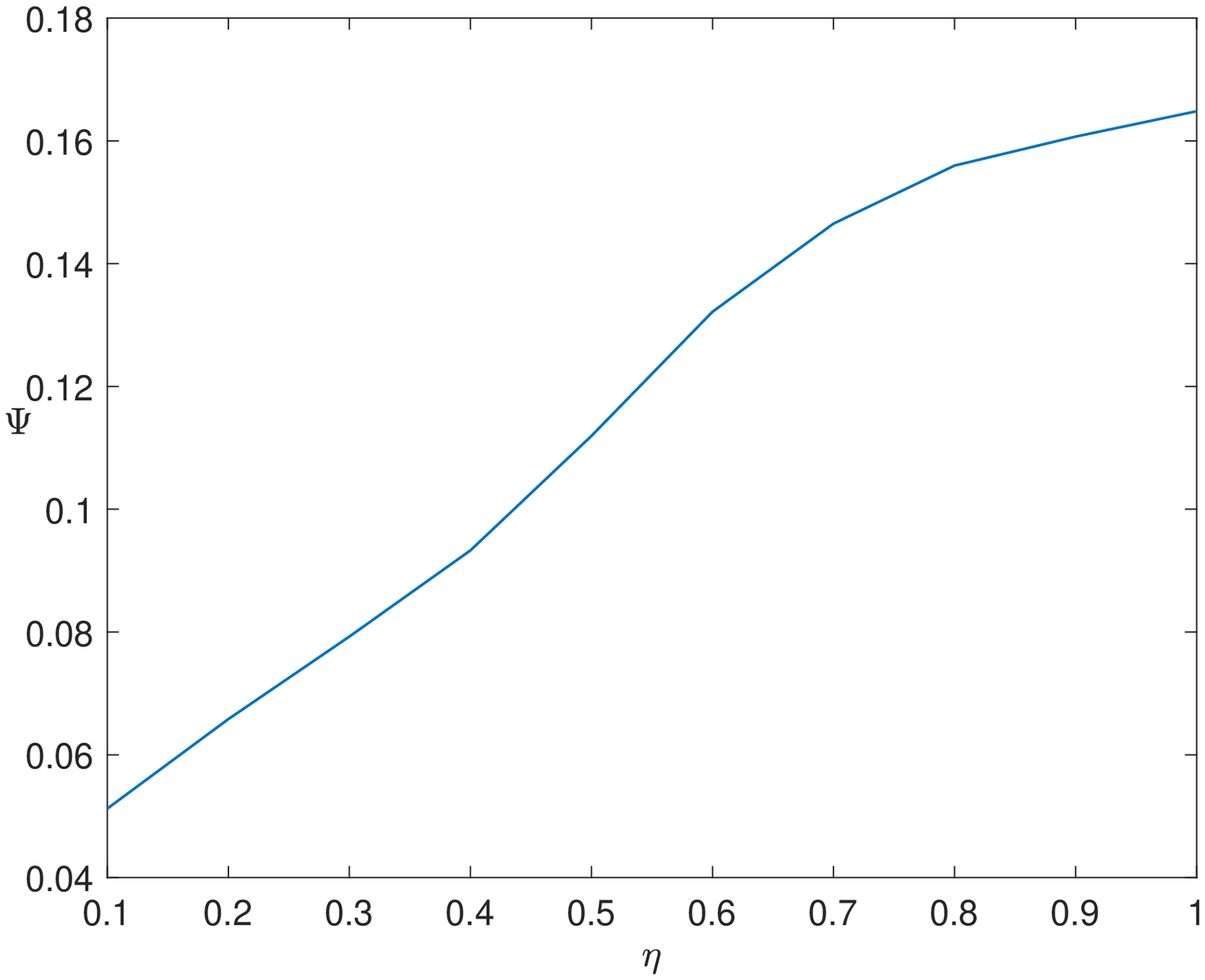}
  \caption{Functionals J~\eqref{J} (left) and $\Psi$~\eqref{functPsi} (right)
    with $m=3$, $\delta=0$ and $\eta\in [0.1, 1]$.}
  \label{fig:eta}
\end{figure}
\begin{figure}[!h]
  \centering
  \includegraphics[width=0.5\textwidth, trim = 15 5 35 20, clip=true]{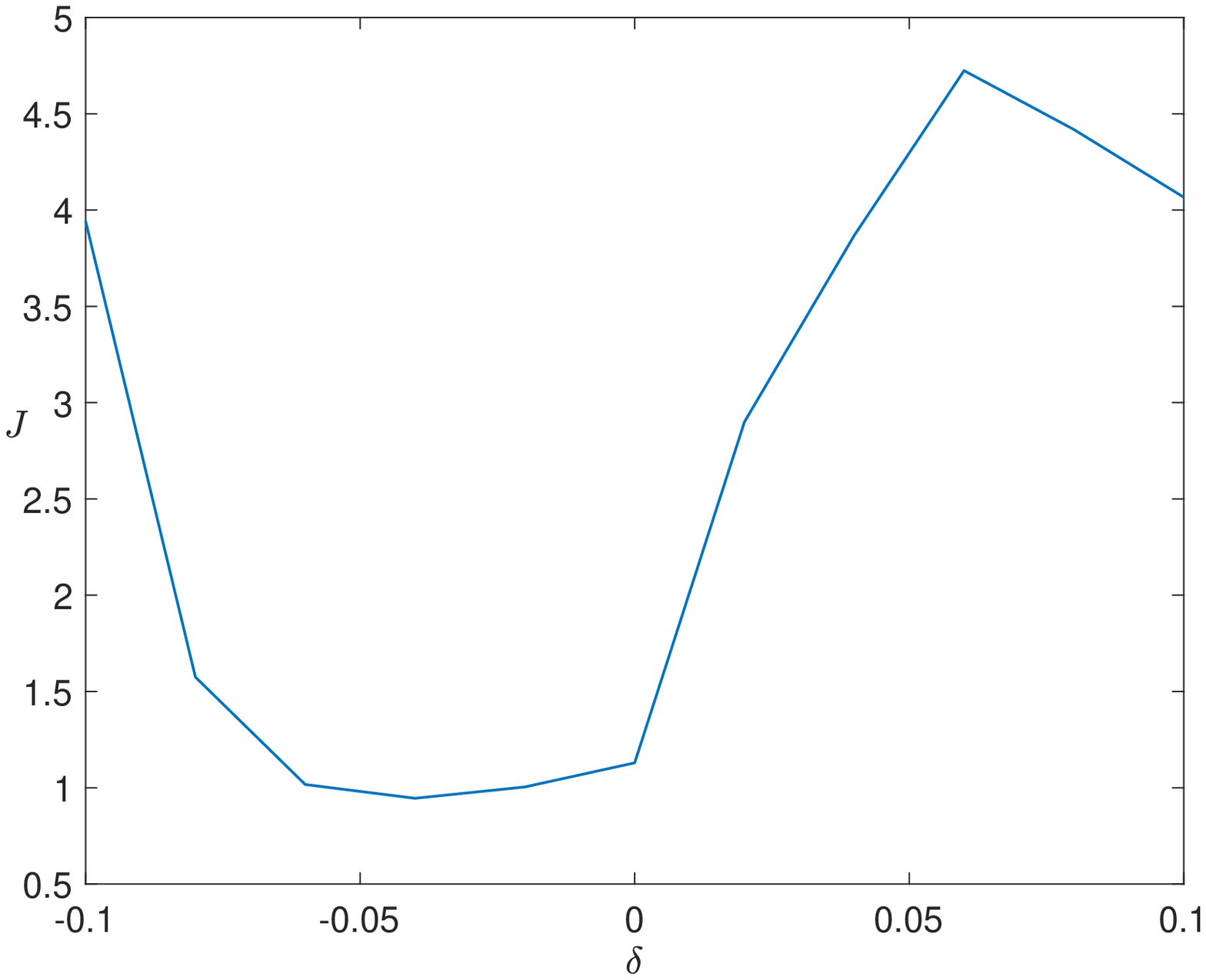}%
  \includegraphics[width=0.5\textwidth, trim = 15 5 35 20, clip=true]{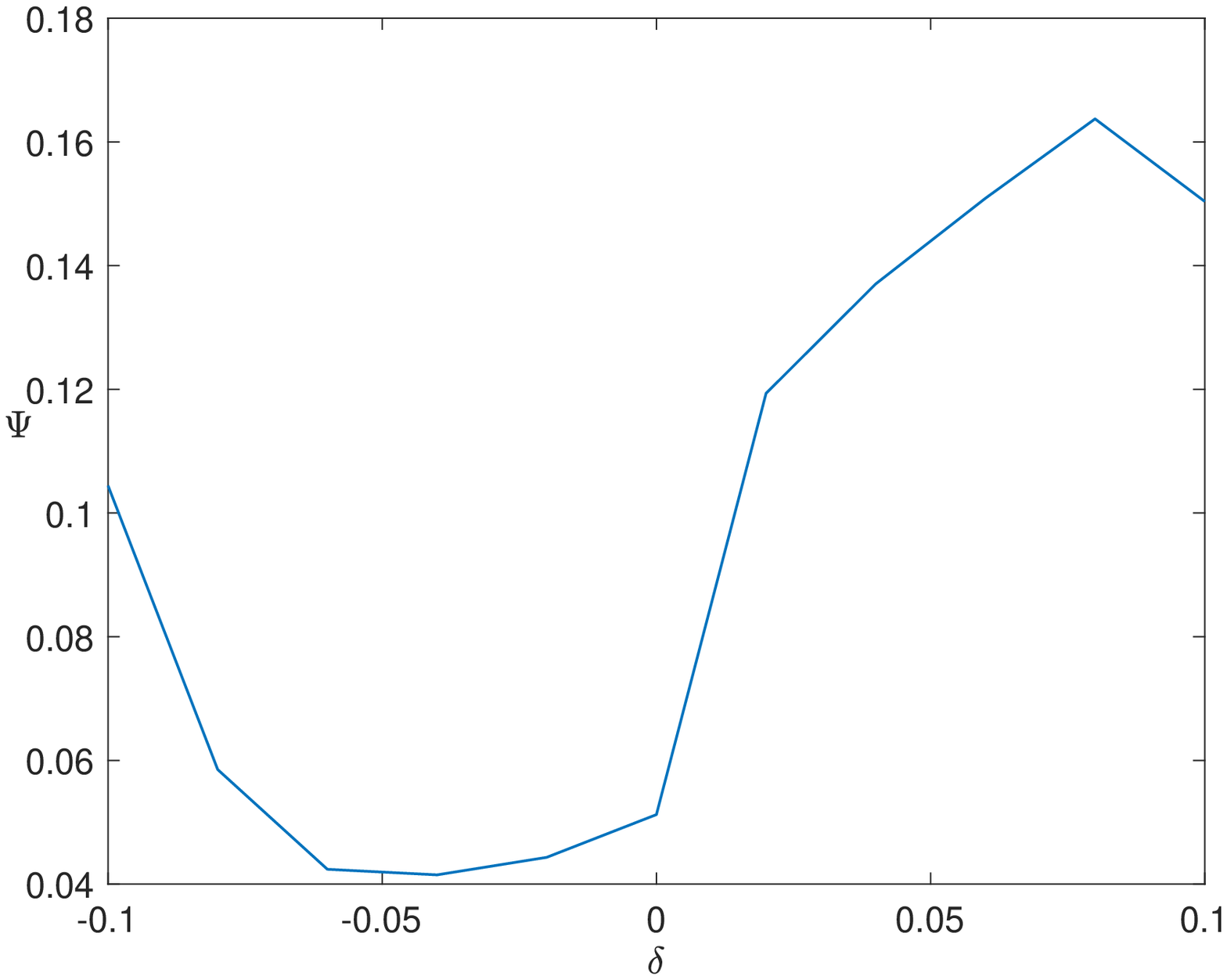}
  \caption{Functionals $J$~\eqref{J} (left) and
    $\Psi$~\eqref{functPsi} (right) with $\eta=0.1$, $m=3$ and
    $\delta\in [-\eta, \eta]$.}
  \label{fig:delta}
\end{figure}
\begin{figure}[!h]
  \centering
  \includegraphics[width=0.5\textwidth, trim = 15 5 35 20, clip=true]{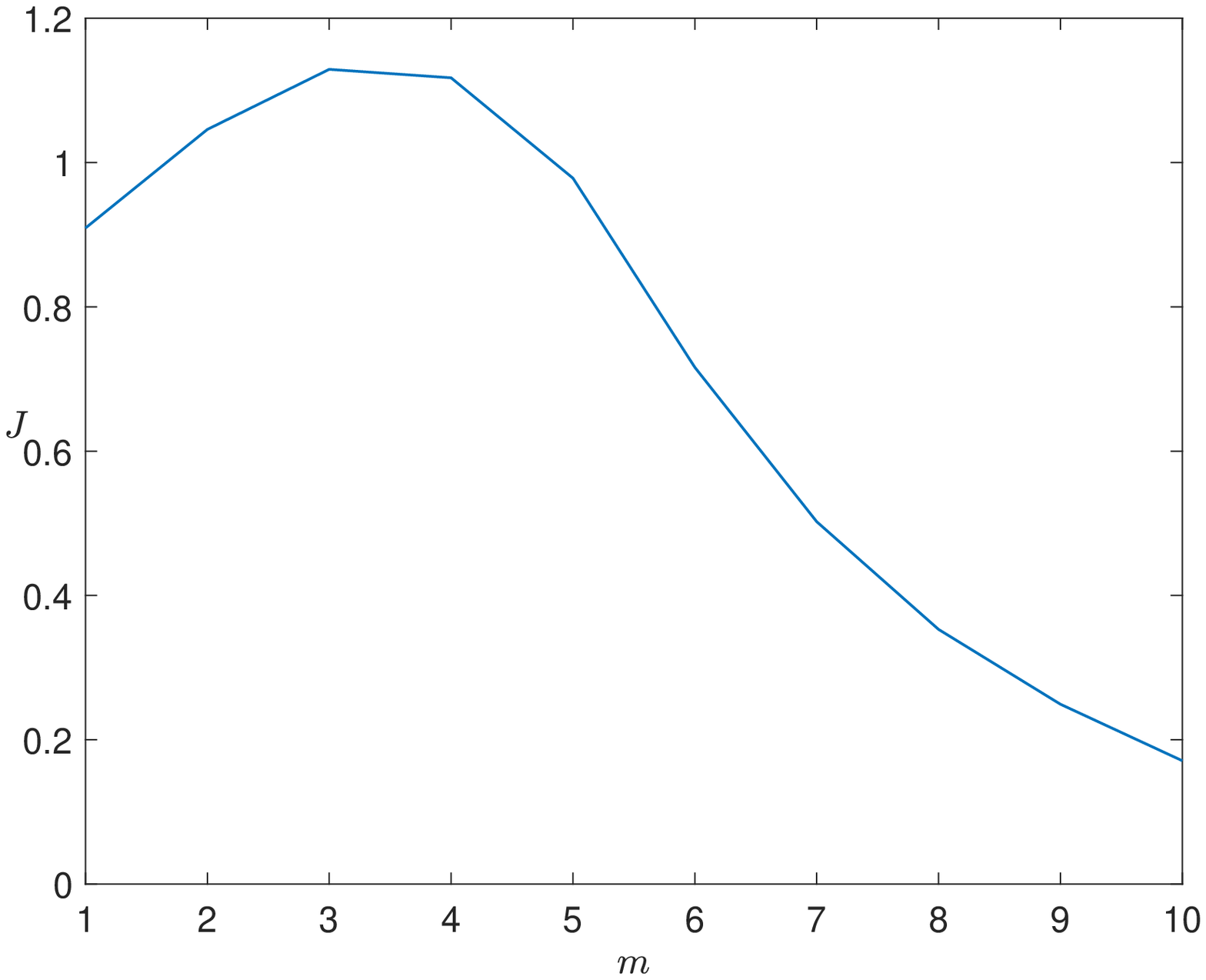}%
  \includegraphics[width=0.5\textwidth, trim = 15 5 35 20, clip=true]{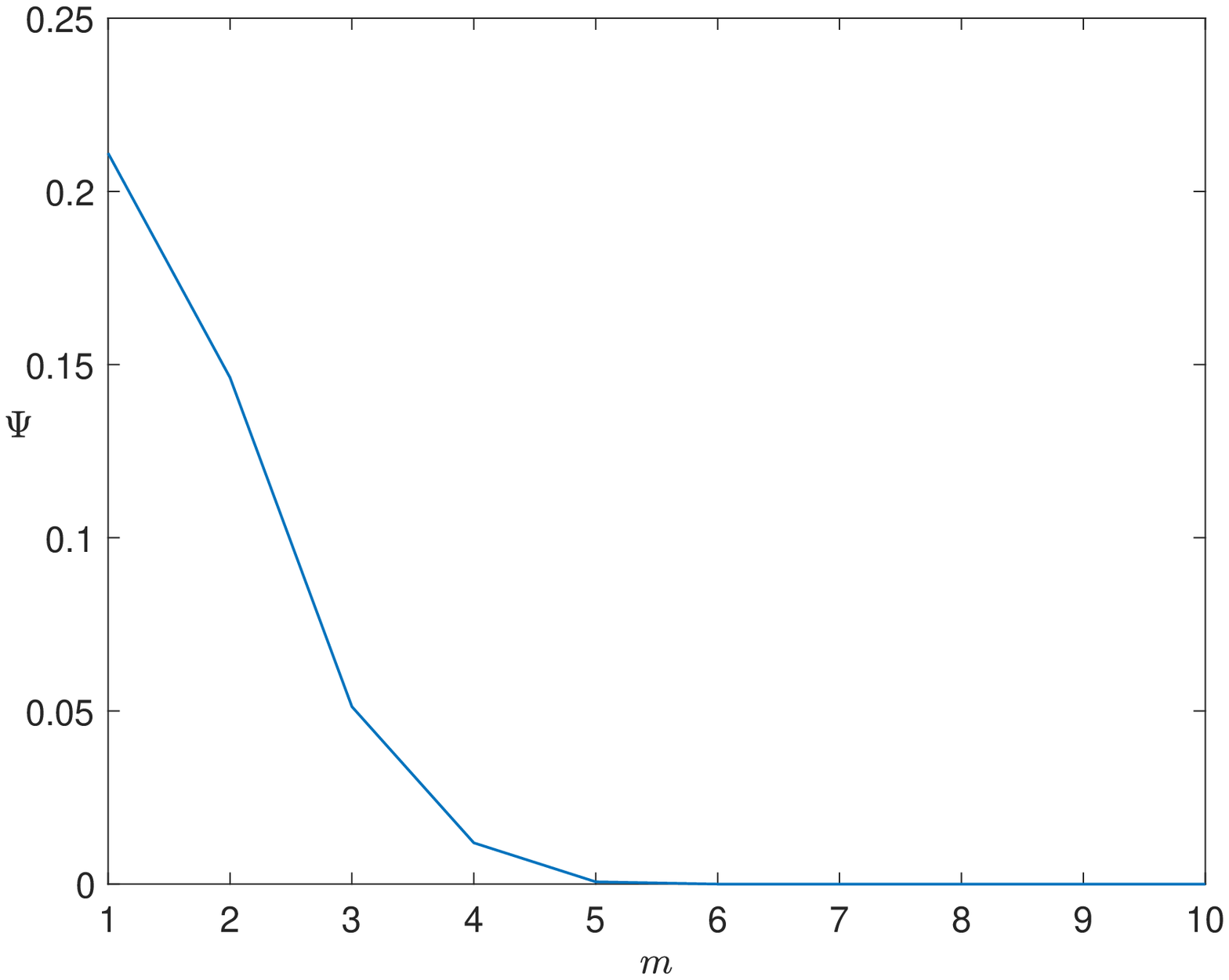}
  \caption{Functionals $J$~\eqref{J} (left) and
    $\Psi$~\eqref{functPsi} (right) with $\eta=0.1$, $\delta=0$ and
    $m\in [1, 10]$.}
  \label{fig:m}
\end{figure}

\begin{figure}[!h]
  \centering
   \begin{subfigure}[t]{0.3\textwidth}
    \includegraphics[width=\textwidth, trim = 15 5 35 20, clip=true]{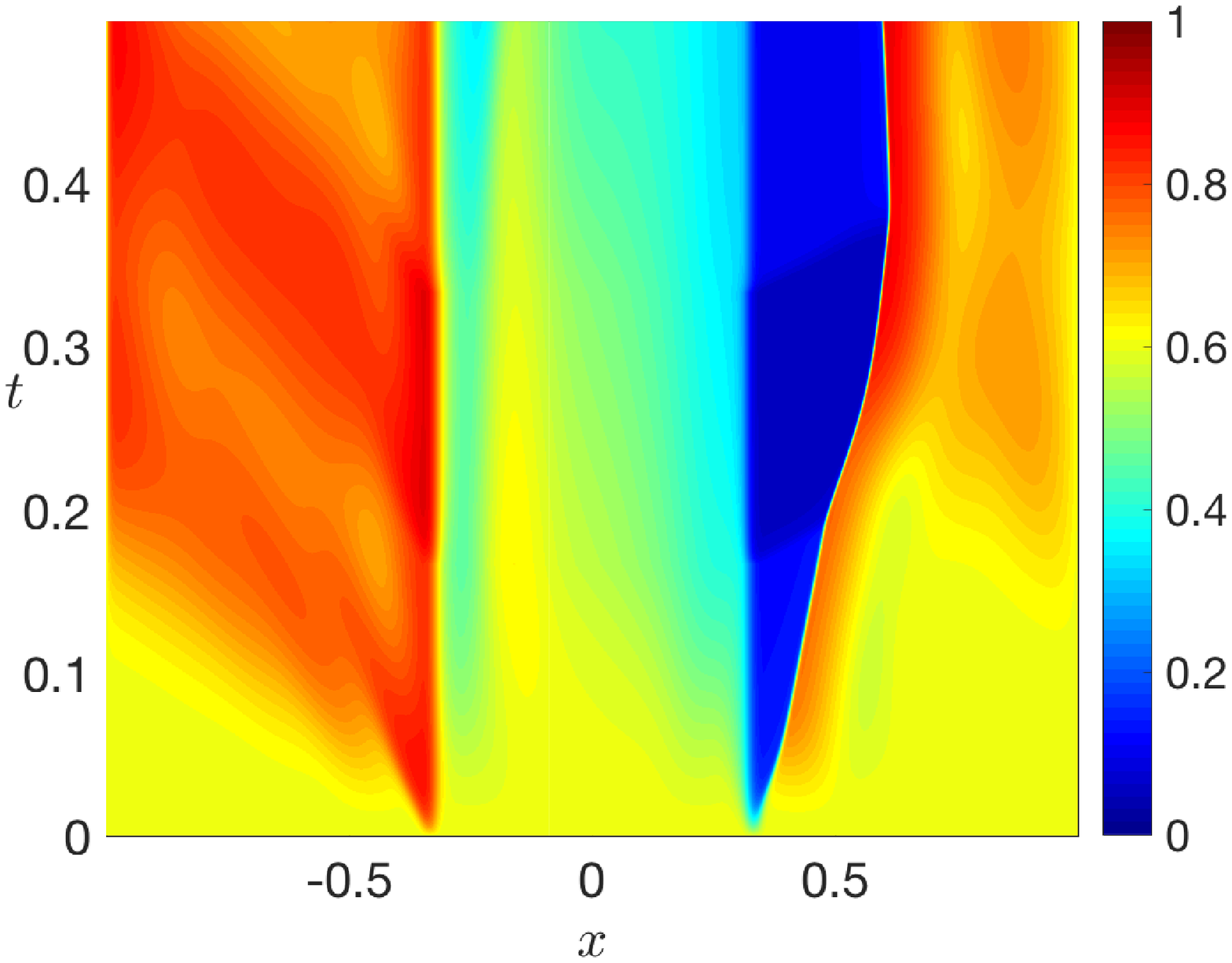}
    \caption{}
    \label{fig:eta02m3}
  \end{subfigure}\quad
  \begin{subfigure}[t]{0.3\textwidth}
    \includegraphics[width=\textwidth, trim = 15 5 35 20, clip=true]{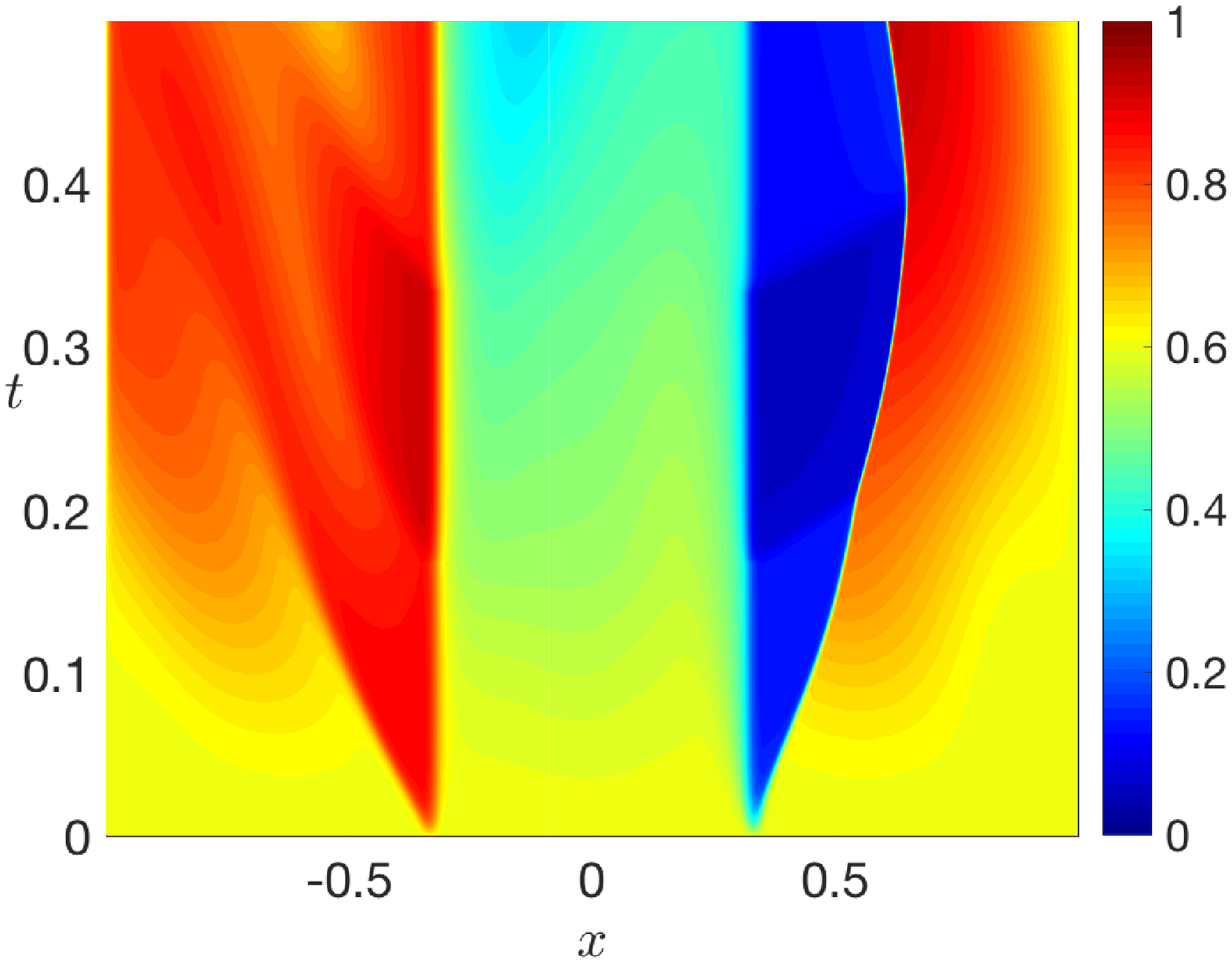}
    \caption{}
    \label{fig:eta05m3}
  \end{subfigure}
  \quad
  \begin{subfigure}[t]{0.3\textwidth}
    \includegraphics[width=\textwidth, trim = 15 5 35 20, clip=true]{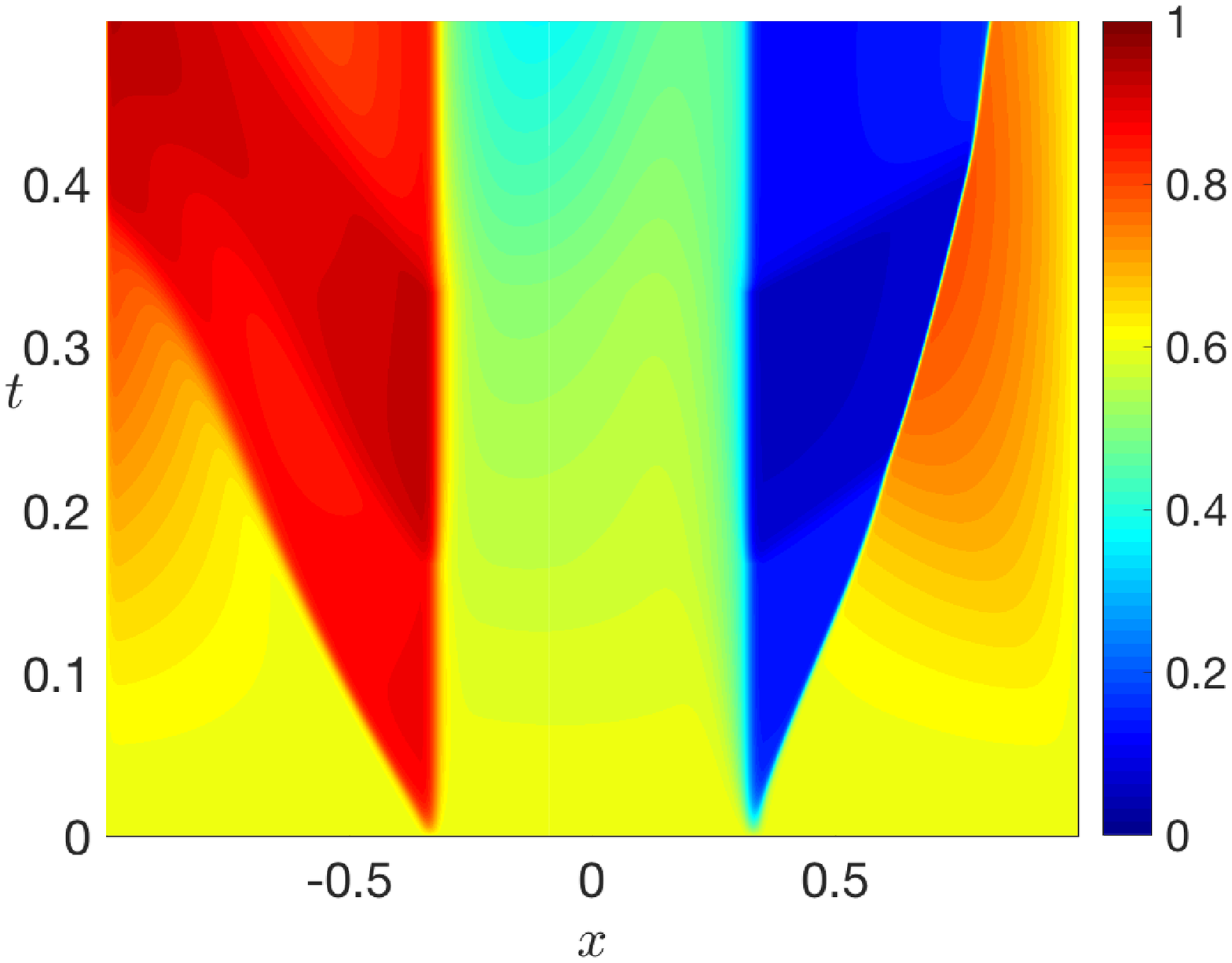}
    \caption{}
    \label{fig:eta1}
  \end{subfigure}
  \caption{$(t,x)$-plots of the solution
    to~\eqref{eq:sample}--\eqref{eq:kernel}, for $m=3$ and $\delta=0$,
    and, from the left, $\eta=0.2, \, 0.5, \, 1$.}
\label{fig:etaEx}
\end{figure}

\begin{figure}[!h]
  \centering
  \begin{subfigure}[t]{0.3\textwidth}
    \includegraphics[width=\textwidth, trim = 15 5 35 20, clip=true]{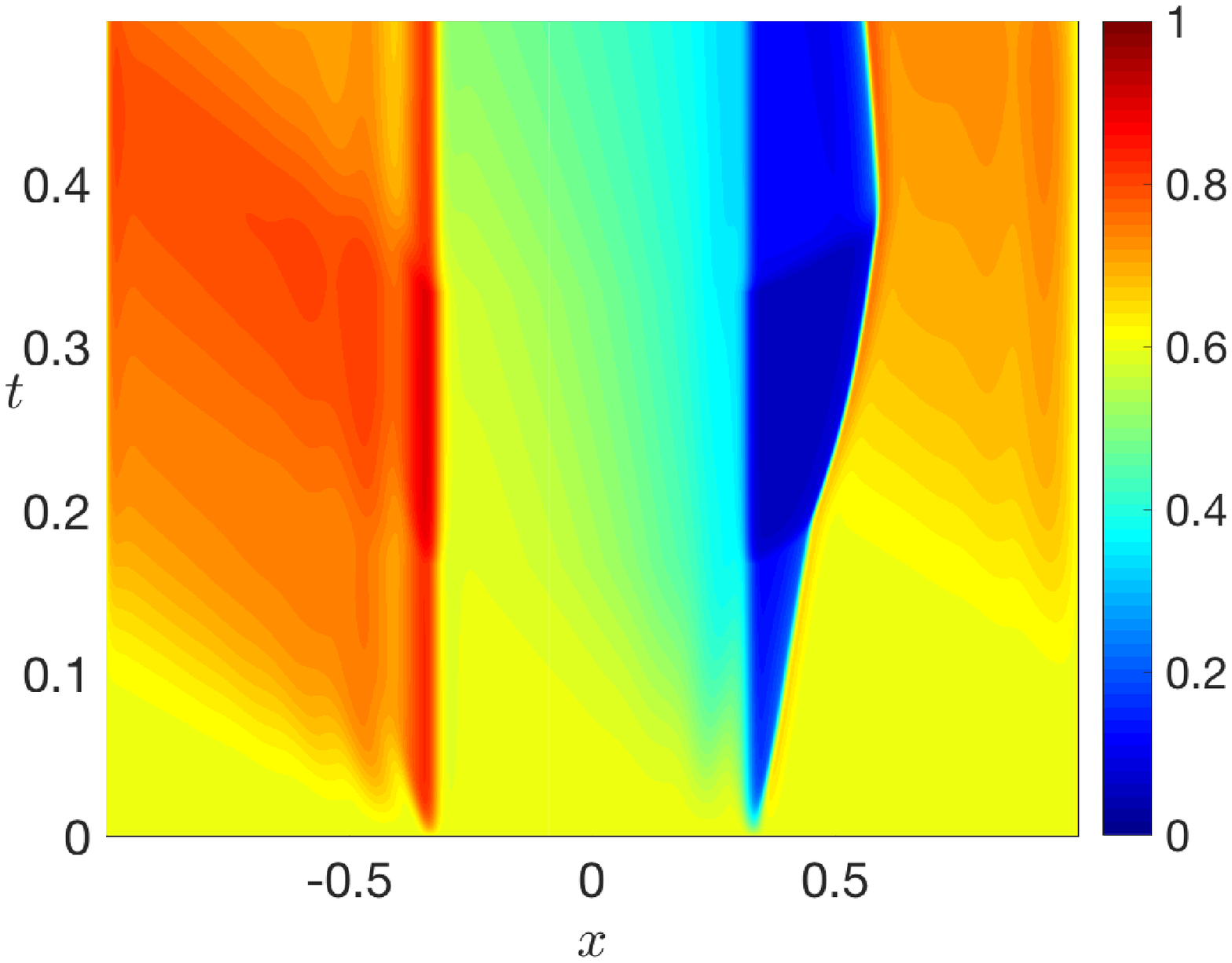}
    \caption{}
    \label{fig:delta04}
  \end{subfigure}\quad
  \begin{subfigure}[t]{0.3\textwidth}
    \includegraphics[width=\textwidth, trim = 15 5 35 20, clip=true]{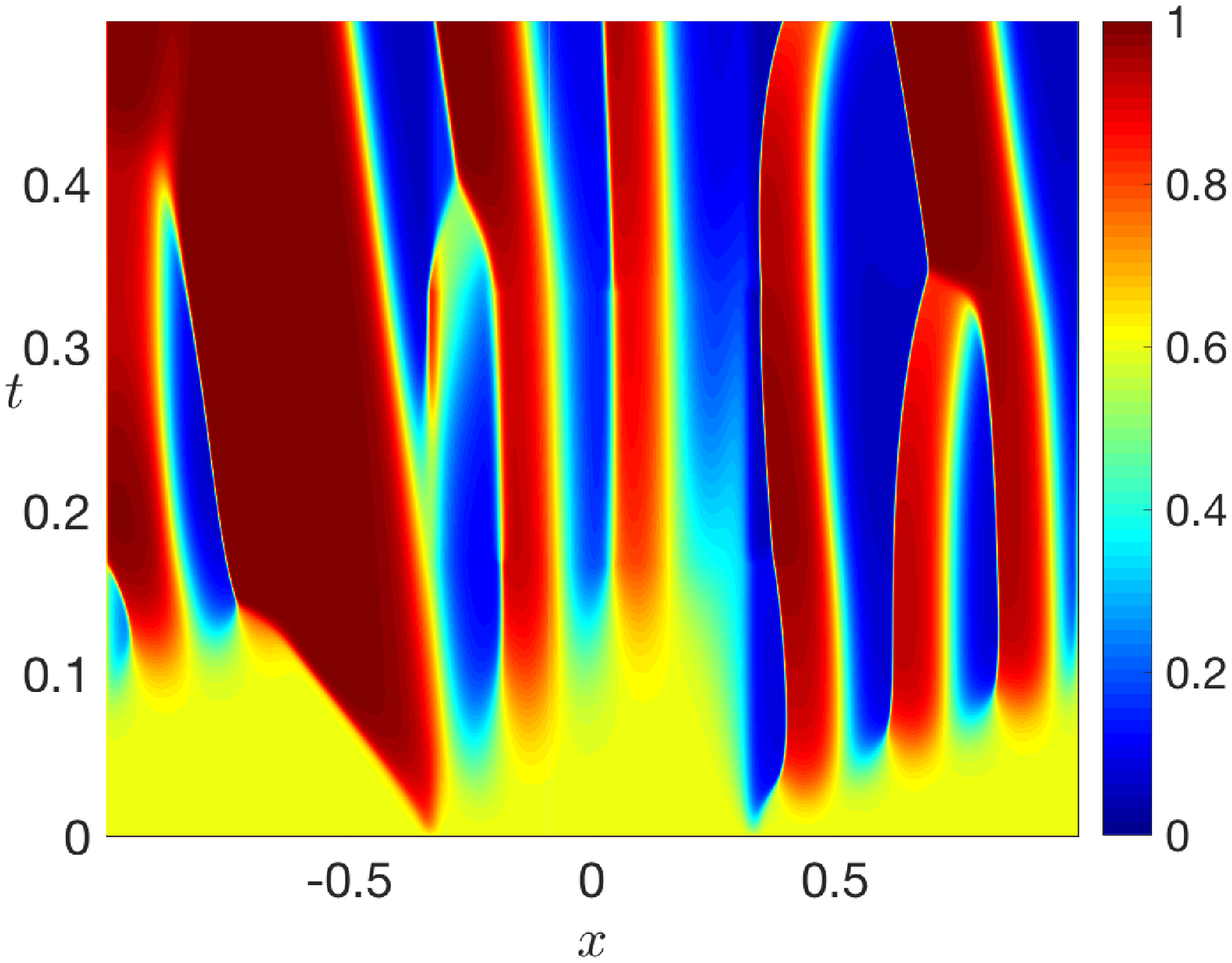}
    \caption{}
    \label{fig:delta06}
  \end{subfigure}\quad
  \begin{subfigure}[t]{0.3\textwidth}
    \includegraphics[width=\textwidth, trim = 15 5 35 20, clip=true]{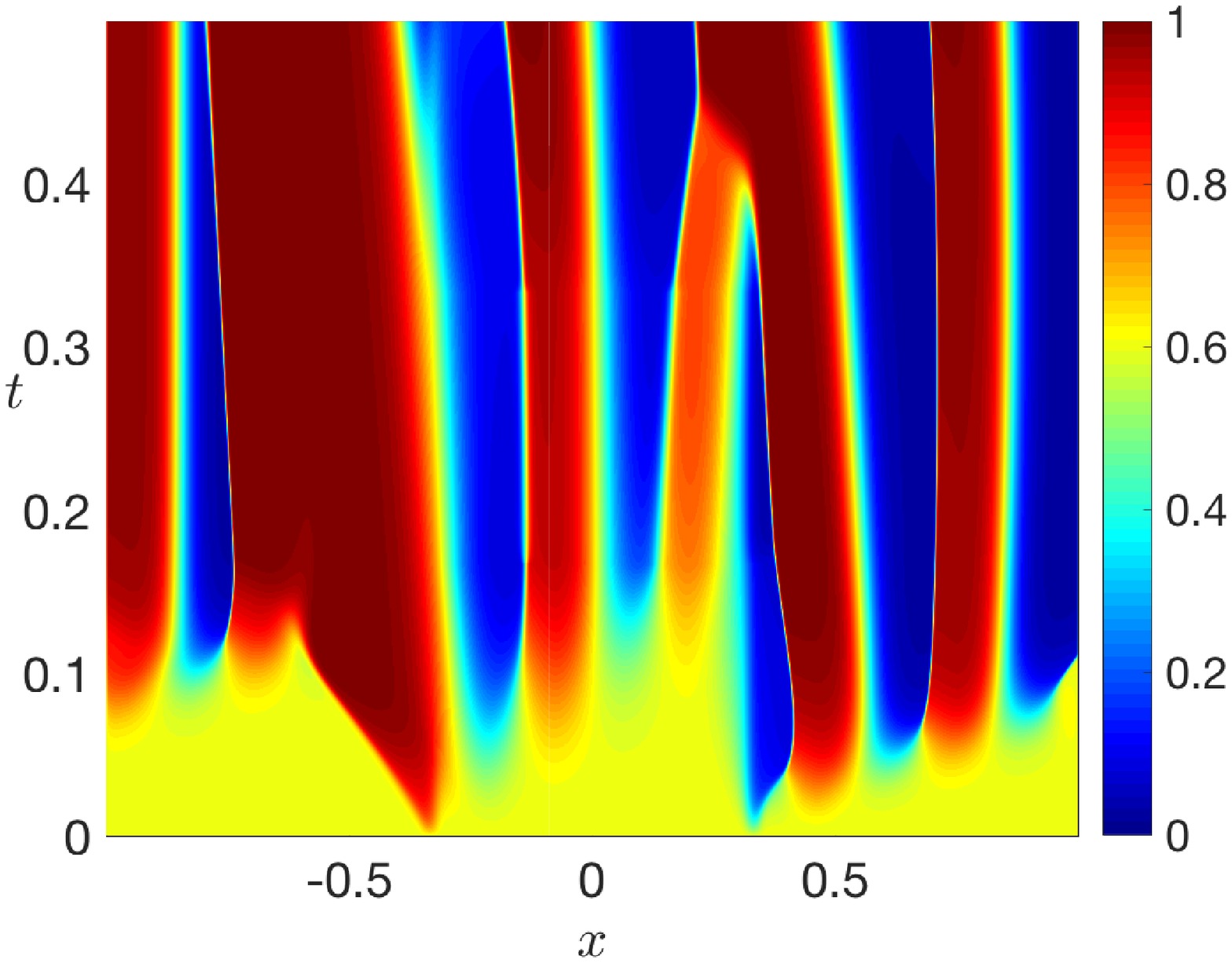}
    \caption{}
    \label{fig:delta10}
  \end{subfigure}
  \caption{$(t,x)$-plots of the solution
    to~\eqref{eq:sample}--\eqref{eq:kernel}, for $m=3$ and $\eta=0.1$,
    and, from the left, $\delta=-0.04, \, 0.06, \, 0.08$.}
\label{fig:deltaEx}
\end{figure}

\begin{figure}[!h]
  \centering
  \begin{subfigure}[t]{0.3\textwidth}
 \includegraphics[width=\textwidth, trim = 15 5 35 20, clip=true]{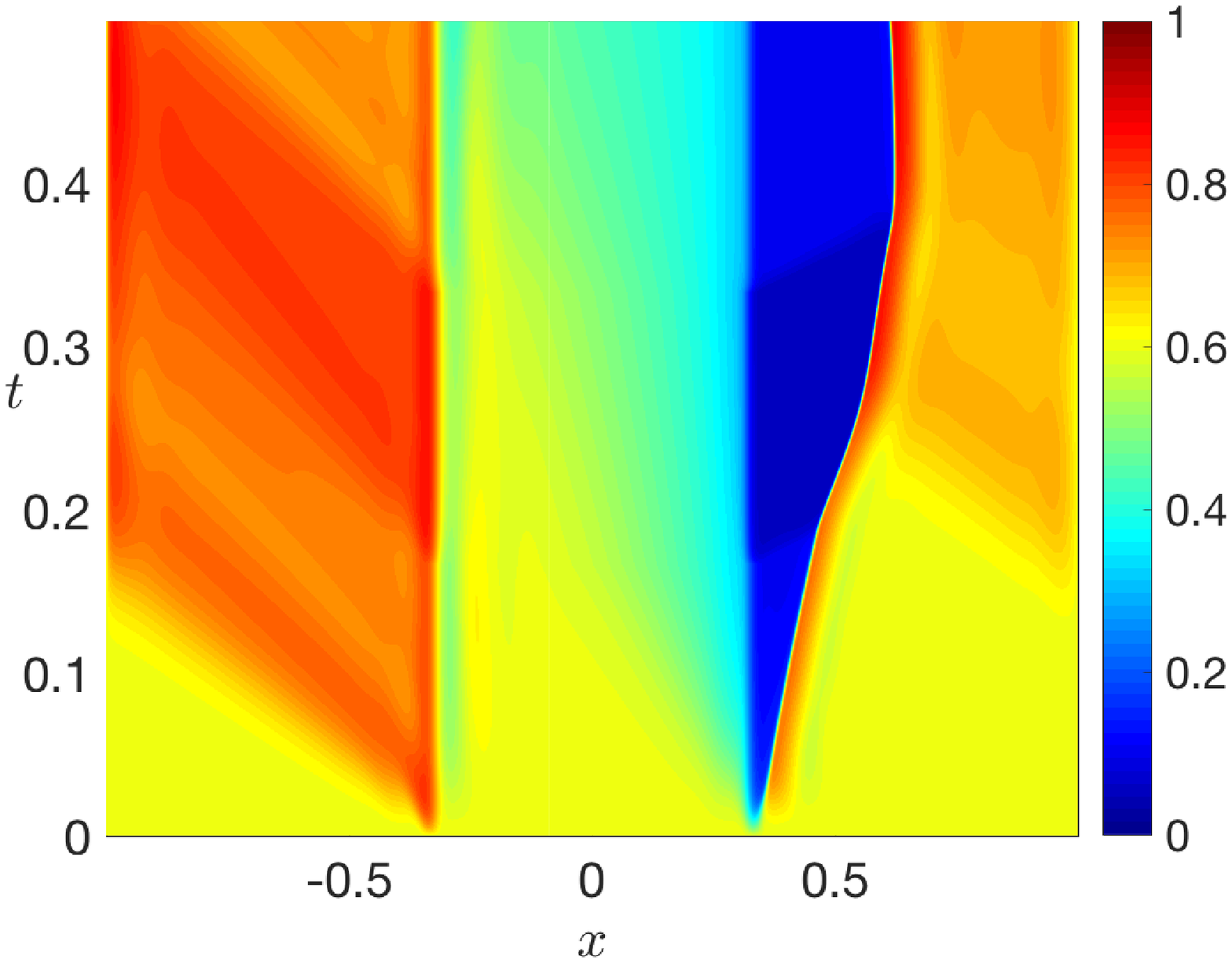}
    \caption{}
    \label{fig:eta01m3}
  \end{subfigure}  \quad
  \begin{subfigure}[t]{0.3\textwidth}
    \includegraphics[width=\textwidth, trim = 15 5 35 20, clip=true]{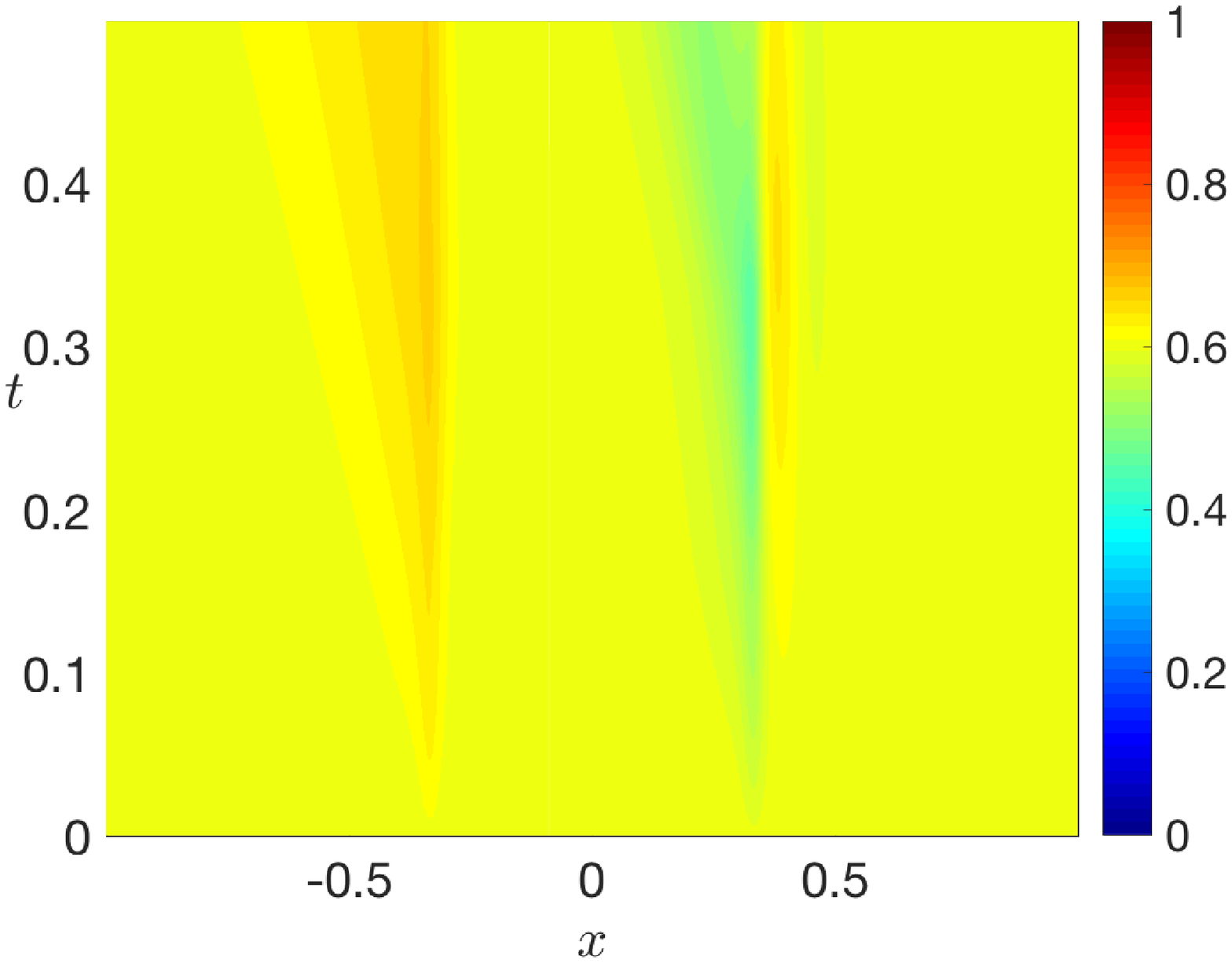}
    \caption{}
    \label{fig:eta01m10}
  \end{subfigure}
   \caption{$(t,x)$-plots of the solution
    to~\eqref{eq:sample}--\eqref{eq:kernel}, for $\eta=0.1$, $\delta=0$,
    and $m = 3$ on the left, $m=10$ on the right.}
  \label{fig:mEx}
\end{figure}

{ \small

  \bibliography{nonlocal}

\def\ocirc#1{\ifmmode\setbox0=\hbox{$#1$}\dimen0=\ht0 \advance\dimen0
  by1pt\rlap{\hbox to\wd0{\hss\raise\dimen0
  \hbox{\hskip.2em$\scriptscriptstyle\circ$}\hss}}#1\else {\accent"17 #1}\fi}
\begin{thebibliography}{10}

\bibitem{ACG2015}
A.~Aggarwal, R.~M. Colombo, and P.~Goatin.
\newblock Nonlocal systems of conservation laws in several space dimensions.
\newblock {\em SIAM J. Numer. Anal.}, 53(2):963--983, 2015.

\bibitem{AHP2016}
D.~Amadori, S.-Y. Ha, and J.~Park.
\newblock On the global well-posedness of {BV} weak solutions to the
  {K}uramoto-{S}akaguchi equation.
\newblock {\em J. Differential Equations}, 262(2):978--1022, 2017.

\bibitem{AmadoriShen2012}
D.~Amadori and W.~Shen.
\newblock An integro-differential conservation law arising in a model of
  granular flow.
\newblock {\em J. Hyperbolic Differ. Equ.}, 9(1):105--131, 2012.

\bibitem{AmorimColomboTexeira}
P.~Amorim, R.~Colombo, and A.~Teixeira.
\newblock On the numerical integration of scalar nonlocal conservation laws.
\newblock {\em ESAIM M2AN}, 49(1):19--37, 2015.

\bibitem{Betancourt2011}
F.~Betancourt, R.~B{\"u}rger, K.~H. Karlsen, and E.~M. Tory.
\newblock On nonlocal conservation laws modelling sedimentation.
\newblock {\em Nonlinearity}, 24(3):855--885, 2011.

\bibitem{BlandinGoatin2016}
S.~Blandin and P.~Goatin.
\newblock Well-posedness of a conservation law with non-local flux arising in
  traffic flow modeling.
\newblock {\em Numer. Math.}, 132(2):217--241, 2016.

\bibitem{ChiarelloGoatin}
F.~A. Chiarello and P.~Goatin.
\newblock {Global entropy weak solutions for general non-local traffic flow
  models with anisotropic kernel}.
\newblock {\em ESAIM Math. Model. Numer. Anal.}, to appear.

\bibitem{ColomboGaravelloMercier2012}
R.~M. Colombo, M.~Garavello, and M.~L\'ecureux-Mercier.
\newblock A class of nonlocal models for pedestrian traffic.
\newblock {\em Mathematical Models and Methods in Applied Sciences},
  22(04):1150023, 2012.

\bibitem{ColomboGoatinRosini2011}
R.~M. Colombo, P.~Goatin, and M.~D. Rosini.
\newblock On the modelling and management of traffic.
\newblock {\em ESAIM Math. Model. Numer. Anal.}, 45(5):853--872, 2011.

\bibitem{ColomboGroli2004}
R.~M. Colombo and A.~Groli.
\newblock Minimising stop and go waves to optimise traffic flow.
\newblock {\em Appl. Math. Lett.}, 17(6):697--701, 2004.

\bibitem{ColomboMercierRosini}
R.~M. Colombo, M.~Mercier, and M.~D. Rosini.
\newblock Stability and total variation estimates on general scalar balance
  laws.
\newblock {\em Commun. Math. Sci.}, 7(1):37--65, 2009.

\bibitem{ColomboRossiBordo}
R.~M. Colombo and E.~Rossi.
\newblock Rigorous estimates on balance laws in bounded domains.
\newblock {\em Acta Math. Sci. Ser. B Engl. Ed.}, 35(4):906--944, 2015.

\bibitem{ColomboRossi2017}
R.~M. Colombo and E.~Rossi.
\newblock {IBVPs} for scalar conservation laws with time discontinuous fluxes.
\newblock {\em Mathematical Methods in the Applied Sciences}, To appear.

\bibitem{Gottlich2014}
S.~G{\"o}ttlich, S.~Hoher, P.~Schindler, V.~Schleper, and A.~Verl.
\newblock Modeling, simulation and validation of material flow on conveyor
  belts.
\newblock {\em Applied Mathematical Modelling}, 38(13):3295 -- 3313, 2014.

\bibitem{Keimer2014}
M.~Gr{\"o}schel, A.~Keimer, G.~Leugering, and Z.~Wang.
\newblock Regularity theory and adjoint-based optimality conditions for a
  nonlinear transport equation with nonlocal velocity.
\newblock {\em SIAM J. Control Optim.}, 52(4):2141--2163, 2014.

\bibitem{KarlsenRisebro}
K.~H. Karlsen and N.~H. Risebro.
\newblock On the uniqueness and stability of entropy solutions of nonlinear
  degenerate parabolic equations with rough coefficients.
\newblock {\em Discrete Contin. Dyn. Syst.}, 9(5):1081--1104, 2003.

\bibitem{KeimerPflug2017}
A.~Keimer and L.~Pflug.
\newblock Existence, uniqueness and regularity results on nonlocal balance
  laws.
\newblock {\em J. Differential Equations}, 263(7):4023--4069, 2017.

\bibitem{Kruzkov}
S.~N. Kru{\v{z}}kov.
\newblock First order quasilinear equations with several independent variables.
\newblock {\em Mat. Sb. (N.S.)}, 81 (123):228--255, 1970.

\bibitem{MercierV2}
M.~{L{\'e}cureux-Mercier}.
\newblock {Improved stability estimates on general scalar balance laws}.
\newblock {\em ArXiv e-prints}, July 2013.

\bibitem{Perthame_book2007}
B.~Perthame.
\newblock {\em Transport equations in biology}.
\newblock Frontiers in Mathematics. Birkh\"auser Verlag, Basel, 2007.

\bibitem{SopasakisKatsoulakis2006}
A.~Sopasakis and M.~A. Katsoulakis.
\newblock Stochastic modeling and simulation of traffic flow: asymmetric single
  exclusion process with {A}rrhenius look-ahead dynamics.
\newblock {\em SIAM J. Appl. Math.}, 66(3):921--944 (electronic), 2006.

\end{thebibliography}

  \bibliographystyle{abbrv}

}

\end{document}